\documentclass[11pt]{article}

\usepackage{amsmath,amssymb,bm,booktabs,array,mathtools,graphicx,url}
\usepackage{amsthm}
\usepackage[margin=1in]{geometry}
\usepackage{microtype}
\usepackage{comment}
\usepackage{placeins}

\numberwithin{equation}{section}

\newtheorem{theorem}{Theorem}[section]
\newtheorem{lemma}[theorem]{Lemma}
\newtheorem{proposition}[theorem]{Proposition}
\newtheorem{corollary}[theorem]{Corollary}
\theoremstyle{remark}
\newtheorem{remark}[theorem]{Remark}

\newcommand{\R}{\mathbb{R}}
\newcommand{\C}{\mathbb{C}}
\newcommand{\Pp}{\mathbb{P}}
\newcommand{\ee}{\mathrm{e}}
\newcommand{\Rea}{\operatorname{Re}}
\newcommand{\Ima}{\operatorname{Im}}

\begin{document}

\title{Uniform Chebyshev asymptotics for repeated-pole rational approximation of the exponential}
\author{Fei Xue\thanks{School of Mathematical and Statistical Sciences, Clemson University, Clemson, South Carolina 29634, USA. Email: fxue@clemson.edu} \and Tianqi Zhang\thanks{School of Data Sciences, Zhejiang University of Finance and Economics, 18 Xueyuan Street, Hangzhou, Zhejiang 310018, China. Email: tianqz@zufe.edu.cn}}
\date{}

\maketitle

\begin{abstract}
We study uniform approximation of $\exp(tz)$, $t>0$, on $(-\infty,0]$ by rational functions $P_m(z)/(q_m-z)^m$ with a prescribed repeated pole $q_m>0$. A M\"obius transformation reduces the problem to polynomial approximation of $F_\lambda(x)=\exp\!\left(-\lambda\frac{1-x}{1+x}\right)$ on $[-1,1]$, where $\lambda=tq_m$. We derive a uniform two-saddle asymptotic formula for the Chebyshev coefficients, including explicit amplitude and phase and a relative remainder for each localized complex saddle contribution, covering fixed, sublinear, and linear pole scalings away from saddle coalescence. Because the two saddle contributions can cancel in a single coefficient, we pass to a growing block of neighboring coefficients and prove that the whole block cannot cancel. This transfers the coefficient asymptotics to approximation errors. For $q_m=(\alpha/t)m$, $0<\alpha<3\sqrt3/2$, the best uniform error has two-sided order $m^{-1/2}H_e(\alpha)^m$; at the optimal ratio $\alpha=1/\sqrt2$ this becomes $m^{-1/2}(\sqrt2-1)^m$. The normalized Chebyshev-weighted $L^2$ projection error has an explicit bounded oscillatory profile. These prefactor-resolved estimates yield a two-term precision-to-work law and quantify mismatch between pole-design and stopping degrees. Finally, the scalar error gives an exact worst-case matrix-action benchmark for self-adjoint negative semidefinite matrices and dimension-independent shift-and-invert Krylov bounds.
\end{abstract}
\par
\medskip

\section{Introduction}
\label{sec:intro}

The basic problem is to approximate $u(t)=\ee^{tA}b$ when $A=A^*\le0$ may have eigenvalues far down the negative real axis. Polynomial approximation must resolve the whole spectral interval, which becomes increasingly demanding as that interval grows under mesh refinement. A rational method can instead place a pole on the positive side of the scalar $z$-plane and repeatedly apply the same shifted inverse. When all poles coincide, one factorization or one preconditioner can be reused. This computational economy is the reason for studying the restricted class $P_m(z)/(q_m-z)^m$, $z\le0$.
Such repeated-pole constructions appear in matrix-exponential restricted-denominator methods; see \cite{Saad1992,EshofHochbruck2006,Druskin2009,Guettel2013,MoretNovati2004,Novati2011,RostamiXue2018}.

This problem should not be confused with unrestricted rational minimax approximation, where all poles may move independently in the complex plane. The latter has its own classical theory, including \cite{CodyMeinardusVarga1969,Aptekarev2002}.  Here the repeated pole is the constraint and also the computational advantage.  Saff, Sch\"onhage, and Varga constructed an explicit geometrically convergent repeated-pole family for $\ee^{-x}$ on $[0,\infty)$~\cite{SaffSchonhageVarga1976}; real-pole restrictions were studied further by Lau~\cite{Lau1977}, Kaufman and Taylor~\cite{KaufmanTaylor1978}, and Borwein~\cite{Borwein1983}.

For the concentrated class, the root-rate question is classical. Andersson's theorem \cite[Sec.~2, p.~86]{Andersson1981} treats arbitrary positive sequences $\eta_m$ in denominators $(x+m\eta_m)^m$ and determines the $m$th-root error through a cubic saddle equation. For a fixed ratio, the optimum is $\eta=1/\sqrt2$ with root factor $\sqrt2-1$, corresponding here to $q_m\sim m/(\sqrt2\,t)$. Our purpose is not to re-prove that root-rate optimum, but to resolve what the root limit suppresses: the prefactor, the oscillatory phase, and the transition from fixed through sublinear to linearly moving poles.

A M\"obius transformation turns the repeated-pole rational problem into ordinary polynomial approximation on $[-1,1]$.  The transformed exponential is $C^\infty$-flat at one endpoint.  This places the problem in the same family as mapped Chebyshev approximation on unbounded intervals studied by Boyd~\cite{Boyd1982,Boyd1987Semi,Boyd1996,Boyd2009}: a fixed map gives a stretched-exponential coefficient scale, whereas a degree-dependent map can produce geometric decay.  Recent work of G\"uttel and Shao \cite{GuettelShao2025,GuettelShao2026} has renewed interest in time-uniform concentrated real-pole families; here we instead resolve coefficient-level saddle asymptotics and prefactor-resolved error laws.

The key is to choose variables in which the fixed-pole limit is not singular. We use $n$ for a Chebyshev coefficient index and $m$ for an approximation degree, and set $\alpha=\lambda/n$ and $\Lambda=(\lambda n^2)^{1/3}$ with $\lambda=tq$. In the corresponding scaled saddle variable, the two relevant complex saddles stay separated and nondegenerate as $\alpha\downarrow0$. The same local geometry therefore covers both a fixed physical pole and a pole proportional to the degree. Returning to the original normalization then recovers Andersson's cubic and saddle action directly.

There are three main steps in the paper. First, we derive a two-saddle Chebyshev coefficient formula whose amplitude, phase, and relative saddle remainder are uniform for $0<\lambda/n\le\alpha_0<3\sqrt3/2$, provided $\Lambda\to\infty$. Second, we turn the oscillatory coefficient information into actual approximation errors. A single coefficient may nearly vanish, so we use a short block of neighboring coefficients and show that the whole block cannot cancel. For the exact linear family this yields the prefactor-resolved two-sided law $E_m(\alpha m)\asymp m^{-1/2}H_e(\alpha)^m$ and an explicit oscillatory profile for the weighted $L^2$ projection error. Third, we invert this two-sided law. At the classical optimum this gives
\[
 m_\varepsilon^*=
 \frac{\log(1/\varepsilon)}{\operatorname{arsinh}(1)}
 -\frac{\log\log(1/\varepsilon)}{2\operatorname{arsinh}(1)}+O(1),
\]
and it also quantifies the penalty when the pole is selected using a nominal degree but the computation stops earlier or later.

The matrix consequences require no new asymptotics.  For self-adjoint negative semidefinite matrices, the spectral theorem makes the scalar half-line error the worst-case matrix-action error for each fixed rational approximant.  The same repeated-pole class generates the shift-and-invert Krylov search space, so the scalar work laws translate into dimension-independent solve counts.

The remainder of the paper is organized as follows. Conceptually, the proof follows \emph{M\"obius map $\to$ uniform saddle formula $\to$ noncancelling coefficient block $\to$ prefactor-resolved error law $\to$ pole/work design}. Section~\ref{sec:mobius} sets up the exact polynomial--rational correspondence. Sections~\ref{sec:uniform-saddle}--\ref{sec:transition} analyze the two saddles and their action, Section~\ref{sec:error} transfers the oscillatory coefficients to approximation errors, and Sections~\ref{sec:design}--\ref{sec:numerics} develop the work laws, matrix consequences, and numerical verification. Details for the two uniform saddle estimates and local robust-design expansion appear in Appendices~\ref{app:saddle-details}--\ref{app:robust-expansion}.

\section{M\"obius reduction and approximation setting}
\label{sec:mobius}

The repeated-pole constraint becomes transparent after one exact change of variables. The M\"obius map below sends the half-line to $[-1,1]$ and turns the restricted rational class into ordinary polynomials of the same degree. Consequently, the scalar approximation problem can be analyzed on a fixed compact interval and then transferred back to the half-line without changing the error.

Fix $t>0$ and write $f_t(z)=\ee^{tz}$ for $-\infty<z\le0$. For a prescribed pole $q>0$ and an integer $m\ge0$, let the repeated-pole or restricted-denominator rational function space be
\begin{equation}\label{eq:Rmq-space-front}
 \mathcal R_m(q)
 :=\left\{\frac{P_m(z)}{(q-z)^m}:\deg P_m\le m\right\}.
\end{equation}
Thus $q$ is the only possible finite pole and has multiplicity at most $m$.  The terminology ``repeated pole'' is used to distinguish this space from rational classes with $m$ freely distributed poles.  In the equivalent positive-half-line problem for $\ee^{-tz}$, $z\ge0$, the pole is located at $-q$.

Consider the M\"obius transformation
\begin{equation}\label{eq:mobius}
 x=\frac{q+z}{q-z},\qquad
 z=-q\frac{1-x}{1+x}.
\end{equation}
It maps $(-\infty,0]$ bijectively onto $(-1,1]$. After adjoining the point $z=-\infty$, the compactified half-line maps onto $[-1,1]$, with $z=0$ corresponding to $x=1$ and $z=-\infty$ to $x=-1$. With $\lambda=tq$, the transformed exponential is
$F_\lambda(x):=f_t(-q(1-x)/(1+x))=\exp(-\lambda(1-x)/(1+x))$ for $-1\le x\le1$, where $F_\lambda(-1)$ is understood as the limiting value $0$.

\begin{proposition}[Exact polynomial--rational isometry]\label{prop:mobius-isometry}
For every $m\ge0$, the map
\begin{equation}\label{eq:restricted-denominator-map}
 p_m(x)\longmapsto
 r_{m,q}(z)
 :=p_m\!\left(\frac{q+z}{q-z}\right)
\end{equation}
is a bijection from $\Pp_m$ onto $\mathcal R_m(q)$.  Moreover,
\begin{equation}\label{eq:mobius-error-isometry}
 \sup_{z\le0}|\ee^{tz}-r_{m,q}(z)|
 =\sup_{-1\le x\le1}|F_\lambda(x)-p_m(x)|.
\end{equation}
Consequently, if
\begin{align}
 E_m(\lambda)
 &:=\inf_{p\in\Pp_m}\|F_\lambda-p\|_\infty,
 \label{eq:best-poly-error-front}\\
 \mathcal E_m^{\rm rat}(t,q)
 &:=\inf_{r\in\mathcal R_m(q)}
   \sup_{z\le0}|\ee^{tz}-r(z)|,
 \label{eq:best-rational-error-front}
\end{align}
then
\begin{equation}\label{eq:rational-polynomial-isometry-front}
 \boxed{\mathcal E_m^{\rm rat}(t,q)=E_m(tq).}
\end{equation}
\end{proposition}

\begin{proof}
For $p_m\in\Pp_m$, clearing the denominator in \eqref{eq:restricted-denominator-map} gives a polynomial in $z$ of degree at most $m$, so the image lies in $\mathcal R_m(q)$. Conversely, substituting the inverse M\"obius map into $P_m(z)/(q-z)^m$ gives a polynomial of degree at most $m$ in $x$; hence the map is bijective. Since \eqref{eq:mobius} bijects $(-\infty,0]$ with $(-1,1]$ and both sides extend continuously to $x=-1$, the sup norms agree. Taking infima proves \eqref{eq:rational-polynomial-isometry-front}.
\end{proof}

The same isometry applies to any target whose M\"obius transform extends continuously to $[-1,1]$; the explicit saddle geometry below is specific to the exponential. The transformed function is analytic for $-1<x\le1$ and is $C^\infty$-flat at $x=-1$: every derivative there vanishes, while the complex continuation has an essential singularity at that endpoint.  This is the source of stretched-exponential Chebyshev decay when $\lambda$ is fixed (and so is the pole $q$), consistent with the general flat-endpoint mechanism studied in \cite{Boyd1982,Boyd1996}.  The crucial point in this study is that $\lambda$ need not remain fixed as the approximation degree grows.

For $n\ge1$, let $a_n(\lambda)$ be the Chebyshev coefficient of $F_\lambda$ under the convention
\begin{equation}\label{eq:cheb-series-front}
 F_\lambda(x)=\frac{a_0(\lambda)}2+
 \sum_{n=1}^\infty a_n(\lambda)T_n(x).
\end{equation}
Thus
\begin{equation}\label{eq:cheb-coef-x}
 a_n(\lambda)
 =\frac{2}{\pi}\int_{-1}^1
 \frac{F_\lambda(x)T_n(x)}{\sqrt{1-x^2}}\,dx.
\end{equation}
With $x=\cos\theta$ and $T_n(\cos\theta)=\cos(n\theta)$, this becomes
\begin{equation}\label{eq:cheb-coef-original}
 a_n(\lambda)=\frac{2}{\pi}\int_0^\pi
 \exp\!\left(-\lambda\tan^2\frac{\theta}{2}\right)
 \cos(n\theta)\,d\theta.
\end{equation}
The coefficient index $n$ and the approximation degree $m$ will be kept distinct.  Section~\ref{sec:uniform-saddle} analyzes sequences $\lambda=\lambda_n$ through the ratio $\lambda_n/n$; Sections~\ref{sec:error}--\ref{sec:design} then apply the result to approximation sequences $\lambda=\lambda_m=tq_m$.  Thus a fixed physical pole corresponds to fixed $\lambda$, whereas a linearly moving pole $q_m=(\alpha/t)m$ corresponds to $\lambda_m/m=\alpha$.  This bookkeeping unifies the fixed-pole and concentrated-pole limits in one asymptotic construction.

\section{Uniform saddle-point analysis of the Chebyshev coefficients}
\label{sec:uniform-saddle}

Two limiting pictures are familiar. With $\lambda$ fixed, the dominant saddle drifts toward the flat endpoint and the coefficient has stretched-exponential decay. When $\lambda$ is proportional to $n$, the saddles remain away from the endpoint and the decay is geometric. Separate saddle calculations would describe both limits but would obscure the transition between them. We rescale before expanding. In the scaled variable, both limits are governed by the same pair of nondegenerate saddles, and the apparently singular limit $\lambda/n\downarrow0$ becomes regular.

\subsection{An exact real-line integral and the scaled phase}
\label{subsec:real-line}

The first step is exact rather than asymptotic: rewrite the coefficient as a real-line integral whose only genuine pole lies above the real axis.  This will let us move the contour through the two lower saddles without crossing a singularity.

\begin{lemma}[Exact real-line representation]\label{lem:real-line}
For every $\lambda>0$ and integer $n\ge1$,
\begin{equation}\label{eq:real-line}
 a_n(\lambda)
 =(-1)^n\frac{2}{\pi}
 \int_{-\infty}^{\infty}
 \frac{\ee^{-\lambda y^2}}{1+y^2}
 \left(\frac{y+i}{y-i}\right)^n\,dy.
\end{equation}
The only pole of the integrand in \eqref{eq:real-line} is at $y=i$. At $y=-i$, the simplified integrand is analytic and nonzero for $n=1$, and has a zero of order $n-1$ for $n\ge2$.
\end{lemma}

\begin{proof}
Set $t=\pi-\theta$ in \eqref{eq:cheb-coef-original}.  Since $\tan((\pi-t)/2)=\cot(t/2)$ and $\cos(n(\pi-t))=(-1)^n\cos(nt)$, we obtain
\[
 a_n(\lambda)=(-1)^n\frac{2}{\pi}\Rea
 \int_0^\pi \ee^{int-\lambda\cot^2(t/2)}\,dt.
\]
With $y=\cot(t/2)$, $dt=-2\,dy/(1+y^2)$ and $\ee^{it}=(y+i)/(y-i)$, and hence
\[
 a_n(\lambda)=(-1)^n\frac{4}{\pi}\Rea
 \int_0^\infty
 \frac{\ee^{-\lambda y^2}}{1+y^2}
 \left(\frac{y+i}{y-i}\right)^n\,dy.
\]
For real $y$, the integrand at $-y$ is the complex conjugate of the integrand at $y$, which gives \eqref{eq:real-line}.  Finally, $\frac{1}{1+y^2}((y+i)/(y-i))^n=(y+i)^{n-1}/(y-i)^{n+1}$, which proves the last assertion.
\end{proof}

For a parameter sequence $\lambda=\lambda_n>0$, introduce
\begin{equation}\label{eq:alpha-delta-Lambda}
 \alpha=\frac{\lambda}{n},\qquad
 \delta=\alpha^{1/3},\qquad
 \Lambda=n\delta=(\lambda n^2)^{1/3}.
\end{equation}
After setting $w=\delta y$, \eqref{eq:real-line} becomes
\begin{equation}\label{eq:scaled-integral}
 a_n(\lambda)
 =(-1)^n\frac{2\delta}{\pi}
 \int_{-\infty}^{\infty}
 \frac{1}{w^2+\delta^2}
 \exp(-\Lambda w^2)
 \left(\frac{w+i\delta}{w-i\delta}\right)^n\,dw.
\end{equation}
To put the scaled integral into standard large-parameter form, we absorb the $n$-dependent rational factor into the exponential. Since $n=\Lambda/\delta$, locally away from $w=\pm i\delta$ we may choose a branch of the logarithm and write
\[
 \left(\frac{w+i\delta}{w-i\delta}\right)^n
 =
 \exp\!\left\{
 \frac{\Lambda}{\delta}
 \log\frac{w+i\delta}{w-i\delta}
 \right\}.
\]
Accordingly, define the phase and amplitude, for $\delta>0$, by
\begin{equation}\label{eq:scaled-phase}
 \Phi_\delta(w)
 =-w^2+\frac{1}{\delta}
 \log\frac{w+i\delta}{w-i\delta},
\end{equation}
and
\begin{equation}\label{eq:amplitude-def}
 A_\delta(w)=\frac{1}{w^2+\delta^2}.
\end{equation}
Then the scaled coefficient integral takes the standard phase--amplitude form
\begin{equation}\label{eq:scaled-phase-form}
 a_n(\lambda)
 =
 (-1)^n\frac{2\delta}{\pi}
 \int_{\mathbb R}
 A_\delta(w)e^{\Lambda\Phi_\delta(w)}\,dw.
\end{equation}
Here the logarithm may be chosen locally along the contour; the exponential is independent of the chosen branch because $\Lambda/\delta=n\in\mathbb N$, so \eqref{eq:scaled-phase-form} agrees exactly with \eqref{eq:scaled-integral}. For later use in the fixed-$\lambda$ limit, we extend both quantities continuously in $\delta$ at $\delta=0$ (for $w\ne0$) by setting
\begin{equation}\label{eq:delta-zero-phase-amplitude}
 \Phi_0(w)=-w^2+\frac{2i}{w},
 \qquad
 A_0(w)=\frac{1}{w^2}.
\end{equation}
Indeed, for fixed $w\ne0$,
\[
 \frac{1}{\delta}
 \log\frac{w+i\delta}{w-i\delta}
 =\frac{2i}{w}+O(\delta^2),
 \qquad \delta \downarrow 0.
\]
Moreover, uniformly for $w$ in compact subsets of $\mathbb C\setminus\{0\}$,
\begin{equation}\label{eq:phase-small-delta-local}
 \Phi_\delta(w)
 =\Phi_0(w)-\frac{2i\delta^2}{3w^3}+O(\delta^4),
 \qquad \delta\downarrow0.
\end{equation}
This extension is used only to make the parameter-uniform estimates transparent; no value $\lambda=0$ is inserted into the original coefficient integral.

Differentiation gives
\begin{equation}\label{eq:phase-derivatives}
 \Phi_\delta'(w)
 =-2w-\frac{2i}{w^2+\delta^2},
 \qquad
 \Phi_\delta''(w)
 =-2+\frac{4iw}{(w^2+\delta^2)^2}.
\end{equation}
Consequently, the stationary points satisfy the cubic equation
\begin{equation}\label{eq:w-cubic}
 w(w^2+\delta^2)=-i.
\end{equation}
The scaling \eqref{eq:alpha-delta-Lambda} has removed the singular motion of the fixed-$\lambda$ saddle: as $\delta\downarrow0$, the roots of \eqref{eq:w-cubic} approach the finite nonzero roots of \(w^3=-i\). In particular, the two relevant lower-half-plane saddles remain \(O(1)\) apart.

\subsection{Saddle geometry and the coalescence threshold}
\label{subsec:saddle-geometry}

We now locate the saddles and the coalescence value.  A real parametrization makes the geometry and later uniformity estimates explicit.

\begin{lemma}[Relevant saddle pair]\label{lem:saddle-pair}
Let
\begin{equation}\label{eq:alpha-critical}
 \alpha_c=\frac{3\sqrt3}{2},
 \qquad
 \delta_c=\alpha_c^{1/3}.
\end{equation}
For every $0\le\delta<\delta_c$, equation \eqref{eq:w-cubic} has exactly two simple roots in the open lower half-plane,
\begin{equation}\label{eq:wpm}
 w_\pm(\delta)=\pm u(\delta)-iv(\delta),
 \qquad u(\delta)>0,\quad v(\delta)>0,
\end{equation}
where $u$ and $v$ are determined by
\begin{equation}\label{eq:uv-relations}
 u^2=3v^2-\delta^2,
 \qquad
 2v(4v^2-\delta^2)=1.
\end{equation}
The third root is $w_0(\delta)=2iv(\delta)$.  Moreover,
\begin{equation}\label{eq:delta-zero-saddles}
 w_\pm(0)=\pm\frac{\sqrt3}{2}-\frac{i}{2},
\end{equation}
and the two lower saddles coalesce precisely at $\delta=\delta_c$.
\end{lemma}

\begin{proof}
Write a lower-half-plane root as $w=u-iv$ with $u,v>0$.  Separating real and imaginary parts in $w^3+\delta^2w=-i$ gives
\[
 u(u^2-3v^2+\delta^2)=0,
 \qquad
 v^3-3u^2v-\delta^2v=-1.
\]
For $u>0$ these equations reduce to \eqref{eq:uv-relations}. The relation $u^2=3v^2-\delta^2$ gives a real value of $u$ precisely when $v\ge\delta/\sqrt3$, with $u>0$ for strict inequality. On this interval, $h_\delta(v):=2v(4v^2-\delta^2)$ is strictly increasing; its minimum occurs at $v=\delta/\sqrt3$ and equals $2\delta^3/(3\sqrt3)$.
Thus there exists a unique solution with $v>\delta/\sqrt3$, and hence $u>0$, if and only if $\delta^3<3\sqrt3/2$.  Symmetry of the cubic gives the second lower root $-u-iv$.  Since the coefficient of \(w^2\) in \(w^3+\delta^2w+i=0\) vanishes, the sum of the three roots is zero by Viète's formula, and the remaining root is $2iv$.  At $\delta=0$, \eqref{eq:uv-relations} gives $v=1/2$ and $u=\sqrt3/2$.  At coalescence the cubic \eqref{eq:w-cubic}, equivalently $p_\delta(w):=w^3+\delta^2w+i=0$, has a multiple root $w_c$. Hence $p_\delta(w_c)=p_\delta'(w_c)=0$, and since $p_\delta'(w)=3w^2+\delta^2$, the coalescing root is purely imaginary; writing $w_c=-iv$ ($v>0$) gives $\delta^2=3v^2$. Substitution into $p_\delta(w_c)=0$ then yields $2v^3=1$. Therefore $v=2^{-1/3}$ and $\delta_c=\sqrt{3}\,2^{-1/3}$. Since $\alpha=\delta^3$, we obtain \(\alpha_c=\frac{3\sqrt{3}}{2}\), which is \eqref{eq:alpha-critical}.
\end{proof}

At a stationary point, \eqref{eq:w-cubic} simplifies the second derivative considerably:
\begin{equation}\label{eq:phase-second-saddle}
 \Phi_\delta''(w_\pm)
 =-6+4i\delta^2w_\pm.
\end{equation}
\begin{equation}\label{eq:real-second-derivative}
 \Rea\Phi_\delta''(w_+)
 =-6+4\delta^2v<0,
 \qquad 0\le\delta<\delta_c,
\end{equation}
and the real part tends to zero only at saddle coalescence.  Hence, for every fixed $\delta_0<\delta_c$, the two saddles remain uniformly nondegenerate for $0\le\delta\le\delta_0$.

It will be useful later to return to the unscaled variable
\begin{equation}\label{eq:zeta-transform}
 \zeta=\frac{iw}{\delta}
 \qquad (\delta>0).
\end{equation}
Equation \eqref{eq:w-cubic} becomes
\begin{equation}\label{eq:andersson-cubic}
 \alpha(\zeta^3-\zeta)+1=0,
\end{equation}
and $w_+$ corresponds to the first-quadrant root $\zeta_\alpha$.  Defining
\begin{equation}\label{eq:psi-alpha}
 \psi_\alpha(\zeta)
 =\log\frac{\zeta-1}{\zeta+1}+\alpha\zeta^2,
\end{equation}
we have the exact identity
\begin{equation}\label{eq:phase-action-identity}
 \Lambda\Phi_\delta(w_+)
 =n\psi_\alpha(\zeta_\alpha).
\end{equation}
Section~4 identifies this action with the classical concentrated-pole rate function.

\subsection{Global contour and uniform local estimates}
\label{subsec:contour}

A local expansion at the two saddles is not enough: we must know that they dominate the \emph{whole} deformed contour, uniformly as $\delta\downarrow0$. Three facts are therefore needed. The real line must be deformable through the two lower saddles without crossing a pole; those saddles must be the only global maxima of the real part of the phase; and the rest of the contour must remain below them by a parameter-uniform gap. A horizontal contour through the saddle pair reduces all three questions to one real variable.

The logarithmic phase has an apparent singularity at $w=-i\delta$, whereas the exact integrand is regular there and vanishes for $n\ge2$.  We therefore use the exact integrand to justify the deformation and use the logarithmic phase only where it is analytic.  This prevents a spurious residue at that point.

\begin{lemma}[Horizontal saddle contour]\label{lem:horizontal-contour}
Fix $0\le\delta_0<\delta_c$.  For every $0<\delta\le\delta_0$, let $w_\pm=\pm u-iv$ be given by Lemma~\ref{lem:saddle-pair}.  The contour in \eqref{eq:scaled-integral} may be deformed from $\R$ to $\R-iv$.  Along this horizontal contour define
\begin{equation}\label{eq:Gdelta}
 G_\delta(x)
 =\Rea\Phi_\delta(x-iv)
 =-x^2+v^2+\frac{1}{2\delta}
 \log\frac{x^2+(v-\delta)^2}{x^2+(v+\delta)^2}.
\end{equation}
The expression in \eqref{eq:Gdelta} is understood in the extended-real sense at its only possible zero: when $v=\delta$ (i.e., $\delta=6^{-1/3}$), set $G_\delta(0)=-\infty$. Then $G_\delta$ has exactly two global maxima, at $x=\pm u$; away from this exceptional logarithmic dip, $x=0$ is a finite local minimum. At every point where $G_\delta$ is finite,
\begin{equation}\label{eq:Gprime}
 G_\delta'(x)
 =x\left[-2+
 \frac{4v}
 {\bigl(x^2+(v-\delta)^2\bigr)
  \bigl(x^2+(v+\delta)^2\bigr)}\right],
\end{equation}
and the bracket is positive for $0<|x|<u$ and negative for $|x|>u$. The functions $G_\delta$ extend continuously to $\delta=0$. With $v(0)=1/2$,
\begin{equation}\label{eq:G0}
 G_0(x)
 =\Rea\Phi_0\!\left(x-\frac{i}{2}\right)
 =-x^2+\frac14-\frac{1}{x^2+1/4}.
\end{equation}
Furthermore, for every sufficiently small fixed $\eta>0$, there exists $c_\eta>0$, depending only on $\delta_0$ and $\eta$, such that
\begin{equation}\label{eq:uniform-gap}
 \sup_{\substack{x\in\R:\\ |x-u|\ge\eta,\ |x+u|\ge\eta}}
 G_\delta(x)
 \le G_\delta(u)-c_\eta
\end{equation}
uniformly for $0\le\delta\le\delta_0$, with $G_0$ understood as in \eqref{eq:G0}.
\end{lemma}

The lemma separates the global issue cleanly. Formula \eqref{eq:Gprime} identifies the two saddle points as the only maxima; compactness on a subcritical parameter interval controls the bounded part of the contour, and the Gaussian factor controls the tails. The one subtlety is the apparent singularity at $-i\delta$: the deformation must be performed with the exact rational integrand, for which this point is regular, rather than with the logarithmic phase alone. The details are in Appendix~\ref{app:saddle-details}.

\begin{remark}[The apparent singularity at $-i\delta$]\label{rem:removable-zero}
When $v>\delta$, the contour deformation passes below the point $w=-i\delta$, and for the isolated value $v=\delta$ the final horizontal line passes through it.  This causes no singular contribution: the exact integrand has the local factor $(w+i\delta)^{n-1}$ there.  In the proof below, the logarithmic phase is used only near $w_\pm$, which stay uniformly separated from $-i\delta$ on every compact subcritical interval.
\end{remark}

\paragraph{Uniform local estimate.}\label{subsec:local-expansion}
Once the global contour is fixed, the remaining calculation is local. The delicate point is uniformity: a fixed-$\delta$ saddle expansion would not by itself justify passage to the fixed-pole limit $\delta\downarrow0$. We therefore keep one neighborhood of the saddle curve on which the phase and amplitude are uniformly analytic and the quadratic curvature is bounded away from zero.

\begin{lemma}[Uniform local saddle estimate]\label{lem:uniform-local-saddle}
Fix $0\le\delta_0<\delta_c$. There exists $\eta>0$, depending only on $\delta_0$, such that the following property holds for $0\le\delta\le\delta_0$. Let $w_+(\delta)=u(\delta)-iv(\delta)$ be the right lower saddle, and fix an even cutoff function $\chi\in C_c^\infty(\R)$ satisfying $0\le\chi\le1$, $\chi(s)=1$ for $|s|\le\eta$, and $\chi(s)=0$ for $|s|\ge2\eta$. With the extensions \eqref{eq:delta-zero-phase-amplitude} at $\delta=0$, define
\begin{equation}\label{eq:Jplus-local}
 J_+(\Lambda,\delta)
 =\int_\R \chi(x-u)
 A_\delta(x-iv)
 \ee^{\Lambda\Phi_\delta(x-iv)}\,dx.
\end{equation}
Then, as $\Lambda\to\infty$,
\begin{equation}\label{eq:uniform-local-asymptotic}
 J_+(\Lambda,\delta)
 =A_\delta(w_+)
 \ee^{\Lambda\Phi_\delta(w_+)}
 \sqrt{\frac{2\pi}{-\Lambda\Phi_\delta''(w_+)}}
 \left(1+R(\Lambda,\delta)\right),
\end{equation}
where the relative remainder $R(\Lambda,\delta)$ satisfies
\begin{equation}\label{eq:uniform-local-remainder}
 |R(\Lambda,\delta)|\le \frac{C}{\Lambda}
\end{equation}
uniformly for $0\le\delta\le\delta_0$, with $C$ independent of $\Lambda$ and $\delta$. The square root is the branch continuous in $\delta$ and positive at $\delta=0$, where $\Phi_0''(w_+(0))=-6$.
\end{lemma}

The mechanism is the standard Gaussian saddle expansion, but the conclusion is stronger than a pointwise-in-$\delta$ estimate. The compact saddle curve stays uniformly separated from the singular points and has uniformly nonzero curvature, so one neighborhood works for every $0\le\delta\le\delta_0$. After Gaussian rescaling, the odd $\Lambda^{-1/2}$ correction integrates to zero against the even cutoff, leaving the uniform relative $O(\Lambda^{-1})$ remainder. Appendix~\ref{app:saddle-details} gives the details.

\subsection{Uniform two-saddle asymptotics}
\label{subsec:uniform-two-saddle}

We can now combine the global contour argument with the local expansion. The two Gaussian neighborhoods produce the conjugate leading terms. Everywhere else the uniform action gap gives exponential decay, except near the apparent singularity, where we return to the exact rational integrand. Only the two saddle neighborhoods contribute at algebraic order; the other pieces are exponentially smaller.

The resulting theorem is uniform all the way down to the fixed-pole limit, but it stops a positive distance before saddle coalescence.  At coalescence the Gaussian saddle model itself changes to an Airy-type one; that is a separate problem and is not needed for either the fixed-pole regime or the classical optimal moving pole.

\begin{theorem}[Uniform two-saddle asymptotics]\label{thm:uniform-coeff}
Fix $0<\alpha_0<\alpha_c=\frac{3\sqrt3}{2}$, and let $\lambda_n>0$ satisfy
\begin{equation}\label{eq:uniform-range}
 0<\alpha_n:=\frac{\lambda_n}{n}\le\alpha_0,
 \qquad
 \Lambda_n:=(\lambda_n n^2)^{1/3}\longrightarrow\infty.
\end{equation}
Set $\delta_n=\alpha_n^{1/3}$, and let $w_+(\delta_n)=u_n-iv_n$ be the right lower saddle in Lemma~\ref{lem:saddle-pair}. Then the Chebyshev coefficients satisfy
\begin{align}
 a_n(\lambda_n)
 ={}&(-1)^n\frac{4\delta_n}{\pi}
 \Rea\Bigg\{
 \frac{\ee^{\Lambda_n\Phi_{\delta_n}(w_+)}}
 {w_+^2+\delta_n^2}
 \sqrt{\frac{2\pi}
 {-\Lambda_n\Phi_{\delta_n}''(w_+)}}
 \Bigg\}
 \label{eq:uniform-leading}\\
 &\quad+O\!\left(
 \delta_n\Lambda_n^{-3/2}
 \ee^{\Lambda_n\Rea\Phi_{\delta_n}(w_+)}
 \right),\nonumber
\end{align}
where the square root is continued from its positive value at $\delta=0$. The implied constant depends only on $\alpha_0$. Equivalently, if $\mathcal C_n$ is the localized complex contribution from the right saddle before taking the real part, then
\begin{equation}\label{eq:complex-relative-error}
 \mathcal C_n
 =\frac{\ee^{\Lambda_n\Phi_{\delta_n}(w_+)}}
 {w_+^2+\delta_n^2}
 \sqrt{\frac{2\pi}
 {-\Lambda_n\Phi_{\delta_n}''(w_+)}}
 \left(1+O(\Lambda_n^{-1})\right)
\end{equation}
uniformly in \eqref{eq:uniform-range}.  We make no relative-error claim for the real coefficient $a_n$, whose conjugate saddle contributions may nearly cancel along subsequences.
\end{theorem}

\begin{proof}
Deform \eqref{eq:scaled-integral} to the horizontal line $\R-iv_n$ by Lemma~\ref{lem:horizontal-contour}.  Choose the cutoff $\chi$ from Lemma~\ref{lem:uniform-local-saddle} around $x=u_n$ and its reflected copy around $x=-u_n$.  The two localized integrals are complex conjugates by the symmetry of the exact integrand.  Lemma~\ref{lem:uniform-local-saddle} therefore gives their sum as twice the real part of the leading contribution in \eqref{eq:uniform-leading}, with a total local error of order
\[
 \Lambda_n^{-3/2}
 \ee^{\Lambda_n\Rea\Phi_{\delta_n}(w_+)}.
\]
After multiplication by the prefactor $2\delta_n/\pi$ in \eqref{eq:scaled-integral}, this is the remainder stated in the theorem.

The complement of the two saddle neighborhoods is exponentially smaller by \eqref{eq:uniform-gap}.  Away from a fixed small neighborhood of $w=-i\delta_n$, the amplitude $A_{\delta_n}$ is uniformly bounded, and hence for some $c>0$ independent of $n$ that part of the remainder is bounded by
\[
 C\ee^{\Lambda_n\Rea\Phi_{\delta_n}(w_+)-c\Lambda_n}.
\]
There is one point at which this amplitude-phase factorization itself needs care.  The final horizontal contour passes through $w=-i\delta$ only when $v(\delta)=\delta$, which by \eqref{eq:uv-relations} occurs at the single value $\delta_*=6^{-1/3}$.  If $\delta_*$ does not belong to $[0,\alpha_0^{1/3}]$, no additional argument is needed.  Otherwise choose a small parameter neighborhood $U$ of $\delta_*$ and a fixed interval $|x|\le\eta_0$ about this regular point, disjoint from the two saddle neighborhoods.  For $\delta\notin U$ the amplitude is again uniformly bounded there.  For $\delta\in U$ we use the exact integrand on $w=x-iv$,
\begin{equation}\label{eq:exact-removable-integrand-bound}
 \ee^{-\Lambda w^2}
 \frac{(w+i\delta)^{n-1}}{(w-i\delta)^{n+1}},
 \qquad n=\frac{\Lambda}{\delta}.
\end{equation}
After shrinking $U$ and $\eta_0$ if necessary, $\delta$ is bounded away from zero, $|w-i\delta|$ is bounded below, and
\[
 \left|\frac{w+i\delta}{w-i\delta}\right|\le r_0<1
 \qquad (\delta\in U,\ |x|\le\eta_0).
\]
As $U$ and $\eta_0$ shrink to $(\delta_*,0)$, the admissible value of $r_0$ tends to zero.  On the other hand, the exact modulus in \eqref{eq:exact-removable-integrand-bound} satisfies, for all sufficiently large $n$,
\[
 \left|
 \ee^{-\Lambda w^2}
 \frac{(w+i\delta)^{n-1}}{(w-i\delta)^{n+1}}
 \right|
 \le C\ee^{\Lambda v(\delta)^2}r_0^{n-1}.
\]
The quantity $v(\delta)^2-\Rea\Phi_\delta(w_+(\delta))$ is bounded on $U$, and $n/\Lambda=1/\delta$ is bounded above and below there.  We may therefore shrink $U$ and $\eta_0$ so that, for a parameter-independent $c>0$,
\[
 C\ee^{\Lambda v(\delta)^2}r_0^{n-1}
 \le C'\ee^{\Lambda[\Rea\Phi_\delta(w_+)-c]}.
\]
Thus the neighborhood of this point is exponentially small as well.  The complete contour remainder is therefore absorbed into the algebraic local remainder, proving \eqref{eq:uniform-leading}.  The complex relative estimate \eqref{eq:complex-relative-error} is exactly \eqref{eq:uniform-local-asymptotic}--\eqref{eq:uniform-local-remainder}.
\end{proof}

The effective large parameter is $\Lambda=(\lambda n^2)^{1/3}$: it has order $n^{2/3}$ for fixed $\lambda$ and order $n$ when \(\lambda\asymp n\). As $\delta\downarrow0$, Lemma~\ref{lem:saddle-pair} and \eqref{eq:delta-zero-phase-amplitude} give
\[
 w_+(\delta)\to w_*:=\frac{\sqrt3}{2}-\frac{i}{2},
 \qquad \Phi_0(w_*)=-\frac32+\frac{3\sqrt3}{2}i.
\]
Since the saddle and phase are smooth in $\delta^2$ and $\Lambda\delta^2=\lambda$, the fixed-$\lambda$ real action is $-\frac32\lambda^{1/3}n^{2/3}+O(1)$, with prefactor of order $\delta\Lambda^{-1/2}=\lambda^{1/6}n^{-2/3}$. Section~\ref{subsec:small-alpha} derives the full fixed-pole formula. The coalescence point itself remains excluded, as noted above.

\section{Pole-scaling transition and the concentrated-pole rate function}
\label{sec:transition}

The saddle theorem was stated in variables chosen to make uniformity transparent. We now return to the variables natural for repeated-pole approximation. The resulting rate function contains the fixed, sublinear, and linear pole regimes in a single formula. For the moment this formula describes the coefficient \emph{envelope}; Section~\ref{sec:error} supplies the noncancellation argument needed to turn that envelope into an approximation-error scale.

\subsection{The saddle action and the concentrated-pole rate function}
\label{subsec:rate-function}

For $0<\alpha<\alpha_c$, let $\zeta_\alpha$ be the first-quadrant root of
\begin{equation}\label{eq:transition-cubic}
 \alpha(\zeta^3-\zeta)+1=0,
\end{equation}
selected in Section~\ref{sec:uniform-saddle}, and let
\begin{equation}\label{eq:transition-psi}
 \psi_\alpha(\zeta)
 =\log\frac{\zeta-1}{\zeta+1}+\alpha\zeta^2.
\end{equation}
The branch of the logarithm is the one inherited continuously from the saddle construction in Section~\ref{sec:uniform-saddle}.  Its real part is branch-independent.  Define
\begin{equation}\label{eq:He-Theta-def}
 H_e(\alpha)
 :=\exp\!\left(\Rea\psi_\alpha(\zeta_\alpha)\right),
 \qquad
 \Theta(\alpha)
 :=\Ima\psi_\alpha(\zeta_\alpha).
\end{equation}
Thus
\begin{equation}\label{eq:He-explicit}
 H_e(\alpha)
 =\exp\!\left(
 \log\left|\frac{\zeta_\alpha-1}{\zeta_\alpha+1}\right|
 +\alpha\Rea(\zeta_\alpha^2)
 \right).
\end{equation}
Formula \eqref{eq:He-explicit}, together with \eqref{eq:transition-cubic}, is exactly Andersson's fixed-ratio rate $\exp G^*(\alpha)$ after identifying his parameter $q$ with $\alpha$; his critical equation $1+\alpha(\zeta^3-\zeta)=0$ is \eqref{eq:transition-cubic} \cite[Secs.~2--4]{Andersson1981}. The same cubic and real action also underlie recent matrix-exponential work \cite{GuettelShao2025,GuettelShao2026}. On the subcritical interval, $\zeta_\alpha$ is the unique root in the open first quadrant, so the root selection agrees unambiguously with the classical one.  In the usual positive-half-line formulation for $\ee^{-tz}$ the corresponding concentrated pole is $-n\alpha/t$; under the $(-\infty,0]$ convention for $\ee^{tz}$ it is $+n\alpha/t$.

The leading coefficient formula from Theorem~\ref{thm:uniform-coeff} takes an especially transparent form in these variables.  Let the square root in
\begin{equation}\label{eq:C-alpha-def}
 C(\alpha)
 :=\frac{4\sqrt\alpha}{\sqrt\pi}\,
 \frac{\zeta_\alpha}{\sqrt{3-2\alpha\zeta_\alpha}}
 =:B(\alpha)\ee^{i\beta(\alpha)}
\end{equation}
be chosen continuously from the positive square root of $3$ as $\alpha\downarrow0$.  Then $B(\alpha)>0$ and $\beta(\alpha)$ is continuous on every compact subinterval of $(0,\alpha_c)$.

\begin{proposition}[Uniform oscillatory coefficient asymptotics]\label{prop:oscillatory-master}
Under the hypotheses of Theorem~\ref{thm:uniform-coeff}, with $\alpha_n=\lambda_n/n$ and $\Lambda_n=(\lambda_n n^2)^{1/3}$,
\begin{equation}\label{eq:oscillatory-master}
 a_n(\lambda_n)
 =(-1)^n n^{-1/2}B(\alpha_n)H_e(\alpha_n)^n
 \left[
 \cos\!\left(n\Theta(\alpha_n)+\beta(\alpha_n)\right)
 +O(\Lambda_n^{-1})
 \right],
\end{equation}
uniformly for $0<\alpha_n\le\alpha_0<\alpha_c$.  The error in brackets is additive; no lower bound on the cosine is assumed.
\end{proposition}

\begin{proof}
From \eqref{eq:zeta-transform}, $w_+=-i\delta\zeta_\alpha$, with $\delta=\alpha^{1/3}$.  The saddle equation gives
\[
 \frac{1}{w_+^2+\delta^2}=iw_+=\delta\zeta_\alpha.
\]
Moreover, since $\psi_\alpha(\zeta)=\delta\Phi_\delta(-i\delta\zeta)$,
\[
 \psi_\alpha''(\zeta_\alpha)
 =-\alpha\Phi_\delta''(w_+).
\]
Differentiating \eqref{eq:transition-psi} and using \eqref{eq:transition-cubic} yields
\begin{equation}\label{eq:psi-second-simplified}
 \psi_\alpha''(\zeta_\alpha)
 =2\alpha(3-2\alpha\zeta_\alpha).
\end{equation}
Consequently,
\[
 -\Lambda\Phi_\delta''(w_+)
 =2n\alpha^{1/3}(3-2\alpha\zeta_\alpha).
\]
Substitution into \eqref{eq:uniform-leading}, together with the exact action identity \eqref{eq:phase-action-identity}, gives
\[
 a_n(\lambda_n)
 =(-1)^n n^{-1/2}
 \Rea\!\left\{
 C(\alpha_n)\ee^{n\psi_{\alpha_n}(\zeta_{\alpha_n})}
 \right\}
 +O\!\left(
 n^{-1/2}B(\alpha_n)H_e(\alpha_n)^n\Lambda_n^{-1}
 \right).
\]
Writing $C=B\ee^{i\beta}$ and separating the real and imaginary parts of the action proves \eqref{eq:oscillatory-master}.
\end{proof}

\begin{remark}[Envelope versus individual coefficients]\label{rem:envelope-coefficient}
The natural nonoscillatory envelope in \eqref{eq:oscillatory-master} is
\begin{equation}\label{eq:coefficient-envelope}
 \mathcal E_n(\lambda_n)
 :=n^{-1/2}B(\alpha_n)H_e(\alpha_n)^n.
\end{equation}
Because the cosine in \eqref{eq:oscillatory-master} may be small along subsequences, asymptotic statements for $\log|a_n|$ require care.  The rate laws below are therefore first stated for the saddle envelope.  Section~\ref{sec:error} will show that these saddle-action rates govern the approximation error for the fixed, algebraically sublinear, and linear pole scalings considered explicitly there; in the general parameter-uniform setting, logarithmic equivalence requires the mild additional condition $\Lambda_m/\log m\to\infty$.
\end{remark}

\subsection{Small-$\alpha$ expansion and scaling transition}
\label{subsec:small-alpha}
\label{subsec:pole-scaling}

The limit $\alpha\downarrow0$ is where the geometric concentrated-pole regime turns into flat-endpoint asymptotics. Write $\delta=\alpha^{1/3}$. In the scaled variable this limit is regular: the right lower saddle is analytic in $\delta^2$ near zero. Expanding \eqref{eq:w-cubic} about \(w_*=\frac{\sqrt3}{2}-\frac{i}{2}\) gives
\begin{equation}\label{eq:w-small-delta}
 w_+(\delta)
 =w_*-\frac{\sqrt3+i}{6}\,\delta^2+O(\delta^6).
\end{equation}
Substitution into the scaled action gives the following expansion.

\begin{lemma}[Small-$\alpha$ action]\label{lem:small-alpha-action}
As $\alpha\downarrow0$,
\begin{align}
 \log H_e(\alpha)
 &=-\frac32\alpha^{1/3}
   +\frac23\alpha
   -\frac1{30}\alpha^{5/3}
   +O(\alpha^{7/3}),
 \label{eq:He-small-alpha}\\
 \Theta(\alpha)
 &=\frac{3\sqrt3}{2}\alpha^{1/3}
   -\frac{\sqrt3}{30}\alpha^{5/3}
   +O(\alpha^{7/3}).
 \label{eq:Theta-small-alpha}
\end{align}
The amplitude and phase shift in \eqref{eq:C-alpha-def} satisfy
\begin{equation}\label{eq:B-beta-small-alpha}
 B(\alpha)
 =\frac{4}{\sqrt{3\pi}}\alpha^{1/6}
  \left(1+O(\alpha^{2/3})\right),
 \qquad
 \beta(\alpha)
 =\frac{\pi}{3}+O(\alpha^{2/3}).
\end{equation}
\end{lemma}

\begin{proof}
The implicit function theorem applied to $w(w^2+\delta^2)=-i$ at $(w,\delta)=(w_*,0)$ gives a convergent expansion in even powers of $\delta$.  To make the first nontrivial coefficients explicit, write
\[
 w_+(\delta)=w_*+c_2\delta^2+c_4\delta^4+O(\delta^6).
\]
Because $w_*^3=-i$, the coefficients of $\delta^2$ and $\delta^4$ in the cubic equation give
\[
 3w_*^2c_2+w_*=0,
 \qquad
 3w_*^2c_4+3w_*c_2^2+c_2=0.
\]
Hence $c_2=-(\sqrt3+i)/6$, while $3w_*c_2^2+c_2=0$ forces $c_4=0$, which proves \eqref{eq:w-small-delta}.  Now use
\[
 \psi_\alpha(\zeta_\alpha)
 =\delta\Phi_\delta(w_+(\delta)).
\]
Expanding the logarithm in \eqref{eq:scaled-phase} and substituting the saddle expansion gives
\begin{equation}\label{eq:psi-small-alpha-complex}
 \psi_\alpha(\zeta_\alpha)
 =\left(-\frac32+\frac{3\sqrt3}{2}i\right)\delta
  +\frac23\delta^3
  -\frac{1+i\sqrt3}{30}\delta^5
  +O(\delta^7).
\end{equation}
Taking real and imaginary parts and using $\delta=\alpha^{1/3}$ proves \eqref{eq:He-small-alpha}--\eqref{eq:Theta-small-alpha}.  Finally, \eqref{eq:C-alpha-def}, \eqref{eq:zeta-transform}, and \eqref{eq:w-small-delta} give
\[
 C(\alpha)
 =\frac{4}{\sqrt{3\pi}}\alpha^{1/6}\ee^{i\pi/3}
 \left(1+O(\alpha^{2/3})\right),
\]
which is equivalent to \eqref{eq:B-beta-small-alpha}.
\end{proof}

The fixed-pole asymptotics now follow as a direct corollary of the uniform theorem rather than from a separate endpoint calculation.

\begin{corollary}[Fixed-$\lambda$ coefficient asymptotics]\label{cor:fixed-lambda-strong}
For every fixed $\lambda>0$,
\begin{align}
 a_n(\lambda)
 ={}&(-1)^n\frac{4\lambda^{1/6}}{\sqrt{3\pi}}\,
 n^{-2/3}
 \exp\!\left(
 -\frac32\lambda^{1/3}n^{2/3}+\frac23\lambda
 \right)
 \nonumber\\
 &\quad\times\left[
 \cos\!\left(
 \frac{3\sqrt3}{2}\lambda^{1/3}n^{2/3}+\frac{\pi}{3}
 \right)
 +O(n^{-2/3})
 \right].
 \label{eq:fixed-lambda-strong}
\end{align}
The implied constant may depend on $\lambda$.
\end{corollary}

\begin{proof}
Set $\alpha=\lambda/n$ in Proposition~\ref{prop:oscillatory-master}.  Equations \eqref{eq:He-small-alpha}--\eqref{eq:B-beta-small-alpha} give
\begin{align*}
 n\log H_e(\lambda/n)
 &=-\frac32\lambda^{1/3}n^{2/3}+\frac23\lambda+O(n^{-2/3}),\\
 n\Theta(\lambda/n)
 &=\frac{3\sqrt3}{2}\lambda^{1/3}n^{2/3}+O(n^{-2/3}),\\
 n^{-1/2}B(\lambda/n)
 &=\frac{4\lambda^{1/6}}{\sqrt{3\pi}}n^{-2/3}
   \left(1+O(n^{-2/3})\right).
\end{align*}
Also $\Lambda^{-1}=O(n^{-2/3})$.  Absorbing all lower-order corrections into the additive term in brackets proves \eqref{eq:fixed-lambda-strong}.
\end{proof}

\paragraph{Scaling consequences.}
The first term in \eqref{eq:He-small-alpha} already shows the transition. If $\lambda_n=o(n)$, the leading action is $-\tfrac32\lambda_n^{1/3}n^{2/3}$; if $\lambda_n$ is proportional to $n$, the full rate function $H_e$ must be retained and the decay becomes geometric. Since $\lambda_n=tq_n$ under \eqref{eq:mobius}, the same classification applies to the repeated pole $q_n$.

\begin{theorem}[Pole-scaling transition]\label{thm:pole-scaling-transition}
Let $\lambda_n>0$ and assume $\alpha_n=\lambda_n/n\le\alpha_0<\alpha_c$ for all sufficiently large $n$.

\begin{enumerate}
\item[(i)] \emph{Sublinear regime.}  Suppose
\begin{equation}\label{eq:sublinear-lambda}
 \lambda_n\sim c n^\gamma,
 \qquad c>0,\qquad 0\le\gamma<1.
\end{equation}
Then the saddle envelope \eqref{eq:coefficient-envelope} satisfies
\begin{equation}\label{eq:sublinear-envelope-rate}
 \lim_{n\to\infty}
 \frac{\log\mathcal E_n(\lambda_n)}
 {n^{(2+\gamma)/3}}
 =-\frac32 c^{1/3}.
\end{equation}
More precisely,
\begin{equation}\label{eq:sublinear-action-two-scale}
 n\log H_e(\alpha_n)
 =-\frac32\lambda_n^{1/3}n^{2/3}
  +\frac23\lambda_n
  -\frac1{30}\lambda_n^{5/3}n^{-2/3}
  +O\!\left(\lambda_n^{7/3}n^{-4/3}\right).
\end{equation}

\item[(ii)] \emph{Linear regime.}  Suppose
\begin{equation}\label{eq:linear-lambda}
 \frac{\lambda_n}{n}\longrightarrow\alpha
 \in(0,\alpha_0].
\end{equation}
Then
\begin{equation}\label{eq:linear-envelope-root-rate}
 \lim_{n\to\infty}
 \mathcal E_n(\lambda_n)^{1/n}
 =H_e(\alpha).
\end{equation}
If $\lambda_n=\alpha n$ exactly, then
\begin{equation}\label{eq:linear-strong-coeff}
 a_n(\alpha n)
 =(-1)^n n^{-1/2}B(\alpha)H_e(\alpha)^n
 \left[
 \cos\!\left(n\Theta(\alpha)+\beta(\alpha)\right)
 +O(n^{-1})
 \right].
\end{equation}
\end{enumerate}
\end{theorem}

\begin{proof}
In the sublinear regime, $\alpha_n\to0$ and
\[
 \Lambda_n=(\lambda_n n^2)^{1/3}
 \sim c^{1/3}n^{(2+\gamma)/3}\to\infty,
\]
so Theorem~\ref{thm:uniform-coeff} applies.  Equation \eqref{eq:sublinear-action-two-scale} is obtained by multiplying \eqref{eq:He-small-alpha} by $n$ and using $\alpha_n=\lambda_n/n$.  The first term has order $n^{(2+\gamma)/3}$, whereas the second has order $n^\gamma$ and is strictly lower order because $\gamma<1$.  The logarithm of the algebraic factor $n^{-1/2}B(\alpha_n)$ is $O(\log n)$, hence negligible on the scale $n^{(2+\gamma)/3}$.  This proves \eqref{eq:sublinear-envelope-rate}.

In the linear regime, $B(\alpha_n)$ stays bounded above and below on a compact subset of $(0,\alpha_c)$ for all sufficiently large $n$, while $H_e(\alpha_n)\to H_e(\alpha)$.  Taking $n$th roots in \eqref{eq:coefficient-envelope} gives \eqref{eq:linear-envelope-root-rate}.  Formula \eqref{eq:linear-strong-coeff} is the specialization of \eqref{eq:oscillatory-master} to fixed $\alpha$, for which $\Lambda_n=\alpha^{1/3}n$.
\end{proof}

For physical poles, substitute $\lambda_n=tq_n$ in Theorem~\ref{thm:pole-scaling-transition}: the sublinear constant becomes $(tc)^{1/3}$ when $q_n\sim cn^\gamma$, and the linear ratio is $tq_n/n$. The corresponding rational-error laws follow in Section~\ref{sec:error}.

\subsection{The optimal linearly moving pole}
\label{subsec:optimal-alpha}

The minimizer of the geometric rate can now be read directly from the saddle action. This is not a new proof of the classical best-approximation result; it explains why the same optimal ratio emerges from the coefficient geometry and provides derivatives needed later for mismatch analysis.

\begin{proposition}[Rate-function identities]\label{prop:He-derivatives}
For $0<\alpha<\alpha_c$,
\begin{equation}\label{eq:He-first-derivative}
 \frac{d}{d\alpha}\log H_e(\alpha)
 =\Rea(\zeta_\alpha^2).
\end{equation}
The function $H_e$ has a unique minimizer in $(0,\alpha_c)$,
\begin{equation}\label{eq:alpha-star}
 \alpha_*=\frac{1}{\sqrt2},
 \qquad
 \zeta_{\alpha_*}=\ee^{i\pi/4},
\end{equation}
and
\begin{equation}\label{eq:He-minimum}
 H_e(\alpha_*)=\sqrt2-1.
\end{equation}
Moreover,
\begin{equation}\label{eq:He-second-at-optimum}
 \left(\log H_e\right)''(\alpha_*)
 =\frac{2\sqrt2}{5}
 =\frac{4}{5\sqrt2}>0.
\end{equation}
In particular, $H_e$ decreases on $(0,\alpha_*)$ and increases on $(\alpha_*,\alpha_c)$, with $0<H_e(\alpha)<1$ throughout $(0,\alpha_c)$.
\end{proposition}

\begin{proof}
Since $\zeta_\alpha$ is a simple stationary point of $\psi_\alpha$, $\partial_\zeta\psi_\alpha(\zeta_\alpha)=0$. Hence $\frac{d}{d\alpha}\psi_\alpha(\zeta_\alpha)=\partial_\alpha\psi_\alpha(\zeta_\alpha)=\zeta_\alpha^2$, which proves \eqref{eq:He-first-derivative} after taking real parts.

To determine its sign, write $w_+=u-iv$, so $\zeta_\alpha=iw_+/\delta=(v+iu)/\delta$ with $\delta=\alpha^{1/3}$. Then, by \eqref{eq:uv-relations},
\begin{equation}\label{eq:Re-zeta-square-v}
 \Rea(\zeta_\alpha^2)
 =\frac{v^2-u^2}{\delta^2}
 =\frac{\delta^2-2v^2}{\delta^2}.
\end{equation}
The second relation in \eqref{eq:uv-relations} can be written as
\begin{equation}\label{eq:delta-v-relation}
 \delta^2=4v^2-\frac{1}{2v}.
\end{equation}
Its right-hand side is strictly increasing for $v>0$, so $v$ increases strictly with $\delta$. Equation \eqref{eq:Re-zeta-square-v} changes sign only when $\delta^2=2v^2$, which with \eqref{eq:delta-v-relation} gives $4v^3=1$. Thus $v=2^{-2/3}$, $\delta=2^{-1/6}$, and $\alpha=2^{-1/2}$. At this point $u=v$, so $\zeta_{\alpha_*}=\ee^{i\pi/4}$.  This proves uniqueness of the stationary point and the asserted monotonicity.

At $\alpha_*$, $\zeta_{\alpha_*}^2=i$, so the term $\alpha_*\zeta_{\alpha_*}^2$ in the action is purely imaginary.  Thus
\[
 H_e(\alpha_*)
 =\left|
 \frac{\ee^{i\pi/4}-1}
 {\ee^{i\pi/4}+1}
 \right|
 =\sqrt2-1,
\]
which proves \eqref{eq:He-minimum}.

Finally, implicit differentiation of \eqref{eq:transition-cubic} yields
\begin{equation}\label{eq:zeta-alpha-derivative}
 \zeta_\alpha'
 =\frac{1}{\alpha^2(3\zeta_\alpha^2-1)}.
\end{equation}
Differentiating \eqref{eq:He-first-derivative} and evaluating at $\alpha_*=1/\sqrt2$, $\zeta_{\alpha_*}=\ee^{i\pi/4}$, gives
\[
 \left(\log H_e\right)''(\alpha_*)
 =\Rea\left(
 \frac{2\zeta_{\alpha_*}}
 {\alpha_*^2(3\zeta_{\alpha_*}^2-1)}
 \right)
 =\frac{2\sqrt2}{5}.
\]
The first-quadrant root tends to $1/\sqrt3$ as $\alpha\uparrow\alpha_c$, so $H_e(0+)=1$ and $H_e(\alpha_c-)=(2-\sqrt3)\ee^{\sqrt3/2}<1$. These limits, together with the monotonicity above, give $H_e<1$ throughout the subcritical interval.
\end{proof}

The optimal ratio and root factor are classical \cite{Andersson1981}. The differential identities above connect them to the present coefficient expansion and will also quantify the effect of degree--pole mismatch in Section~\ref{sec:design}.

\section{From coefficients to approximation errors}
\label{sec:error}

The coefficient formulas of Sections~\ref{sec:uniform-saddle}--\ref{sec:transition} are oscillatory. A single coefficient can therefore be arbitrarily small relative to its saddle envelope, so a coefficient asymptotic cannot simply be relabeled as an approximation-error asymptotic. The error, however, sees a whole tail of Chebyshev modes. We exploit this by combining an exact differential identity for the tail action with a short block of neighboring coefficients that cannot all cancel at once.

\subsection{Chebyshev projection and the tail action}
\label{subsec:error-notation}

Let
\begin{equation}\label{eq:cheb-weighted-norm}
 \|g\|_{2,C}^2
 :=\int_{-1}^1\frac{|g(x)|^2}{\sqrt{1-x^2}}\,dx
\end{equation}
be the Chebyshev-weighted $L^2$ norm, and write
\begin{equation}\label{eq:cheb-projection}
 \Pi_mF_\lambda
 :=\frac{a_0(\lambda)}{2}
   +\sum_{k=1}^m a_k(\lambda)T_k.
\end{equation}
For the remainder
\begin{equation}\label{eq:projection-remainder}
 R_m(\lambda):=F_\lambda-\Pi_mF_\lambda,
\end{equation}
Parseval's identity gives
\begin{equation}\label{eq:parseval-tail}
 \|R_m(\lambda)\|_{2,C}^2
 =\frac{\pi}{2}\sum_{k=m+1}^\infty |a_k(\lambda)|^2,
\end{equation}
and absolute convergence of the Chebyshev series gives
\begin{equation}\label{eq:sup-tail-upper-basic}
 \|R_m(\lambda)\|_\infty
 \le \sum_{k=m+1}^\infty |a_k(\lambda)|.
\end{equation}
The best polynomial error $E_m(\lambda)$ was defined in \eqref{eq:best-poly-error-front}.

To study a coefficient tail with the target parameter $\lambda$ held fixed while the coefficient index varies, extend the saddle action continuously in the index.  For $x>0$ such that $0<\lambda/x<\alpha_c$, set
\begin{equation}\label{eq:continuous-tail-action}
 \mathcal S(x;\lambda)
 :=x\,\psi_{\lambda/x}\!\left(\zeta_{\lambda/x}\right),
 \qquad
 \mathcal A(x;\lambda):=\Rea\mathcal S(x;\lambda).
\end{equation}
Let
\begin{equation}\label{eq:R-rho-omega-def}
 \mathfrak r(\alpha)
 :=\frac{\zeta_\alpha-1}{\zeta_\alpha+1}
 =:\rho(\alpha)\ee^{i\omega(\alpha)},
 \qquad 0<\omega(\alpha)<\pi.
\end{equation}
The argument is chosen continuously on $(0,\alpha_c)$.  Since $\zeta_\alpha$ lies in the open first quadrant, $0<\rho(\alpha)<1$.

\begin{lemma}[Tail-action identities]\label{lem:tail-action-derivative}
For $0<\lambda/x<\alpha_c$,
\begin{equation}\label{eq:tail-action-derivative-complex}
 \frac{\partial}{\partial x}\mathcal S(x;\lambda)
 =\log\mathfrak r(\lambda/x).
\end{equation}
Consequently,
\begin{equation}\label{eq:tail-action-derivative-real}
 \frac{\partial}{\partial x}\mathcal A(x;\lambda)
 =\log\rho(\lambda/x)<0.
\end{equation}
Moreover, as $\alpha\downarrow0$,
\begin{equation}\label{eq:logR-small-alpha}
 \log\mathfrak r(\alpha)
 =(-1+i\sqrt3)\alpha^{1/3}
  +\frac{1+i\sqrt3}{45}\alpha^{5/3}
  +O(\alpha^{7/3}).
\end{equation}
Hence, for every $0<\alpha_0<\alpha_c$, there are constants $c_j,C_j>0$, depending only on $\alpha_0$, such that
\begin{align}
 c_1\alpha^{1/3}
 &\le -\log\rho(\alpha)
 \le C_1\alpha^{1/3},
 \label{eq:rho-comparable}\\
 c_2\alpha^{1/3}
 &\le \sin\omega(\alpha)
 \le C_2\alpha^{1/3},
 \qquad 0<\alpha\le\alpha_0.
 \label{eq:omega-comparable}
\end{align}
The amplitude in \eqref{eq:C-alpha-def} satisfies
\begin{equation}\label{eq:B-comparable}
 c_3\alpha^{1/6}\le B(\alpha)\le C_3\alpha^{1/6},
 \qquad 0<\alpha\le\alpha_0,
\end{equation}
and the phase functions obey
\begin{equation}\label{eq:phase-derivative-bounds}
 |\omega'(\alpha)|\le C_4\alpha^{-2/3},
 \qquad
 |\beta'(\alpha)|\le C_5\alpha^{-1/3}.
\end{equation}
\end{lemma}

\begin{proof}
Put $\nu=\lambda/x$. Since $\zeta_\nu$ is stationary for $\psi_\nu$, $\frac{d}{d\nu}\psi_\nu(\zeta_\nu)=\zeta_\nu^2$. Therefore
\begin{align*}
 \frac{\partial}{\partial x}\mathcal S(x;\lambda)
 &=\psi_\nu(\zeta_\nu)
   -\nu\zeta_\nu^2
 =\log\frac{\zeta_\nu-1}{\zeta_\nu+1},
\end{align*}
which proves \eqref{eq:tail-action-derivative-complex} and \eqref{eq:tail-action-derivative-real}.

To obtain \eqref{eq:logR-small-alpha}, write $g(\alpha)=\psi_\alpha(\zeta_\alpha)$.  From \eqref{eq:psi-small-alpha-complex},
\[
 g(\alpha)
 =\left(-\frac32+\frac{3\sqrt3}{2}i\right)\alpha^{1/3}
  +\frac23\alpha
  -\frac{1+i\sqrt3}{30}\alpha^{5/3}
  +O(\alpha^{7/3}).
\]
The identity just proved may also be written as \(\log\mathfrak r(\alpha)=g(\alpha)-\alpha g'(\alpha)\). Since \(g\) has a convergent local expansion in \(\alpha^{1/3}\), termwise differentiation yields \eqref{eq:logR-small-alpha}. Since \(\mathfrak r(\alpha)\) remains in the open upper half-plane for \(0<\alpha<\alpha_c\), \(\sin\omega(\alpha)>0\). Near zero, \eqref{eq:logR-small-alpha} gives
\[
 -\log\rho(\alpha)=\alpha^{1/3}+O(\alpha^{5/3}),
 \qquad
 \omega(\alpha)=\sqrt3\,\alpha^{1/3}+O(\alpha^{5/3}).
\]
Thus the two quantities in \eqref{eq:rho-comparable}--\eqref{eq:omega-comparable} are both of order $\alpha^{1/3}$ there; continuity on every interval $[\varepsilon,\alpha_0]$ supplies the uniform constants away from zero.  The same argument applied to \eqref{eq:B-beta-small-alpha} gives \eqref{eq:B-comparable}.

For the derivative bounds, differentiating the displayed expansion of $\omega$ gives $\omega'(\alpha)=O(\alpha^{-2/3})$.  Likewise, $\alpha^{-1/6}C(\alpha)$ is analytic in $\alpha^{2/3}$ near zero and has nonzero limit $4e^{i\pi/3}/\sqrt{3\pi}$, so its argument satisfies $\beta(\alpha)=\pi/3+O(\alpha^{2/3})$ and hence $\beta'(\alpha)=O(\alpha^{-1/3})$.  On $[\varepsilon,\alpha_0]$ both derivatives are bounded by smoothness.  This proves \eqref{eq:phase-derivative-bounds}.
\end{proof}

The strict negativity in \eqref{eq:tail-action-derivative-real} is the basic monotonicity needed below: the saddle envelope decreases with the Chebyshev index once the target parameter $\lambda$ has been fixed.  Integrating \eqref{eq:rho-comparable} gives the useful estimate
\begin{equation}\label{eq:action-decay-integrated}
 \mathcal A(k;\lambda)-\mathcal A(m;\lambda)
 \le -c\lambda^{1/3}\left(k^{2/3}-m^{2/3}\right),
 \qquad k\ge m,
\end{equation}
whenever $\lambda/m\le\alpha_0$, with $c>0$ depending only on $\alpha_0$.

\subsection{A short phase block and removal of coefficient cancellation}
\label{subsec:phase-block}

We next show that the oscillatory factor in a single coefficient cannot suppress an entire short block.  The block length is dictated by two competing requirements.  By \eqref{eq:omega-comparable}, the phase advances by about $\omega(\alpha_m)\asymp\delta_m$ per index, so roughly $\delta_m^{-1}$ consecutive modes are needed to sweep an $O(1)$ phase interval.  Over exactly that many indices, \eqref{eq:tail-action-derivative-real} changes the real action by only $O(1)$, so the coefficient envelopes remain comparable.  Thus
\[
 J_m\asymp\delta_m^{-1}=(m/\lambda_m)^{1/3}
\]
is simultaneously long enough to defeat cancellation and short enough not to lose the saddle scale.  This is the point at which the coefficient envelope becomes a genuine approximation-error scale.

For a sequence $\lambda_m>0$, define
\begin{equation}\label{eq:alpha-delta-Lambda-m-error}
 \alpha_m:=\frac{\lambda_m}{m},
 \qquad
 \delta_m:=\alpha_m^{1/3},
 \qquad
 \Lambda_m:=m\delta_m=(\lambda_m m^2)^{1/3}.
\end{equation}
For $k\ge m$, let
\begin{align}
 P(k;\lambda_m)
 &:=k^{-1/2}B(\lambda_m/k),
 \label{eq:tail-prefactor}\\
 \vartheta(k;\lambda_m)
 &:=\Ima\mathcal S(k;\lambda_m)
   +\beta(\lambda_m/k).
 \label{eq:tail-phase}
\end{align}
Then Proposition~\ref{prop:oscillatory-master}, used with coefficient index $k$ and parameter $\lambda_m$, gives
\begin{equation}\label{eq:pairwise-tail-coeff}
 a_k(\lambda_m)
 =(-1)^k P(k;\lambda_m)
   \ee^{\mathcal A(k;\lambda_m)}
 \left[
  \cos\vartheta(k;\lambda_m)
  +O\!\left((\lambda_m k^2)^{-1/3}\right)
 \right]
\end{equation}
uniformly whenever $0<\lambda_m/k\le\alpha_0<\alpha_c$ and $(\lambda_m k^2)^{1/3}$ is large.

\begin{lemma}[Noncancelling phase block]\label{lem:phase-block}
Fix $0<\alpha_0<\alpha_c$.  There exist constants $K,\,c,\,C>0$, depending only on $\alpha_0$, with the following property.  Suppose
\begin{equation}\label{eq:general-tail-sequence-assumption}
 0<\alpha_m=\frac{\lambda_m}{m}\le\alpha_0,
 \qquad
 \Lambda_m=(\lambda_m m^2)^{1/3}\longrightarrow\infty.
\end{equation}
Set
\begin{equation}\label{eq:Jm-def}
 J_m:=\left\lceil{K}/{\delta_m}\right\rceil.
\end{equation}
Then $J_m=o(m)$ and, for all sufficiently large $m$,
\begin{equation}\label{eq:phase-block-cos-square}
 \sum_{j=1}^{J_m}
 \cos^2\vartheta(m+j;\lambda_m)
 \ge cJ_m.
\end{equation}
Moreover, uniformly for $1\le j\le J_m$,
\begin{equation} \label{eq:block-envelope-comparable}
 \frac{\lambda_m^{1/6}
 \ee^{\mathcal A(m;\lambda_m)}}{Cm^{2/3}}
 \le
 P(m+j;\lambda_m)
 \ee^{\mathcal A(m+j;\lambda_m)}
 \le
 \frac{C\lambda_m^{1/6}
 \ee^{\mathcal A(m;\lambda_m)}}{m^{2/3}}.
\end{equation}
Consequently,
\begin{equation}\label{eq:block-square-lower}
 \sum_{j=1}^{J_m}|a_{m+j}(\lambda_m)|^2
 \ge c\,m^{-1}
 \ee^{2\mathcal A(m;\lambda_m)}.
\end{equation}
\end{lemma}

\begin{proof}
Since $J_m/m=O(\Lambda_m^{-1})$, the block is $o(m)$.  On this block, $\lambda_m/(m+j)=\alpha_m(1+O(\Lambda_m^{-1}))$.  From \eqref{eq:B-comparable}, $P(m+j;\lambda_m)\asymp \lambda_m^{1/6}m^{-2/3}$ uniformly. Also, \eqref{eq:rho-comparable} gives $|\mathcal A(m+j;\lambda_m)-\mathcal A(m;\lambda_m)|\le C\delta_mj\le CK$, proving \eqref{eq:block-envelope-comparable}. Let $\omega_m:=\omega(\alpha_m)$.
Differentiating \eqref{eq:tail-phase} and using Lemma~\ref{lem:tail-action-derivative} gives
\[
 \frac{\partial}{\partial x}\vartheta(x;\lambda_m)
 =\omega(\lambda_m/x)
  -\frac{\lambda_m}{x^2}\,
   \beta'(\lambda_m/x).
\]
The second term is $O(\delta_m^2/m)$ near $x=m$. By \eqref{eq:phase-derivative-bounds}, for $m\le x\le m+J_m$ we have $|\omega(\lambda_m/x)-\omega_m|\le C\delta_m|x-m|/m$. Integrating this variation over an interval of length at most $J_m$, and treating the $\beta$ term similarly, gives $O(\delta_mJ_m^2/m+\delta_m^2J_m/m)=O(\Lambda_m^{-1})$. Hence, uniformly for $1\le j\le J_m$,
\begin{equation}\label{eq:block-phase-linearization}
 \vartheta(m+j;\lambda_m)
 =\vartheta(m;\lambda_m)+j\omega_m+O(\Lambda_m^{-1}).
\end{equation}

For an exact arithmetic progression of phases,
\begin{equation}\label{eq:cos-square-sum-identity}
 \sum_{j=1}^{J}\cos^2(\theta+j\omega)
 =\frac{J}{2}
 +\frac{\sin(J\omega)
 \cos(2\theta+(J+1)\omega)}{2\sin\omega}.
\end{equation}
By \eqref{eq:omega-comparable}, $\sin\omega_m\ge c_0\delta_m$.  Choose $K$ in \eqref{eq:Jm-def} so large that $J_m\ge4/(c_0\delta_m)$. Then the right-hand side of \eqref{eq:cos-square-sum-identity} is at least $3J_m/8$.  Since $|\cos^2(x+e)-\cos^2x|\le |e|$, the perturbation in \eqref{eq:block-phase-linearization} changes the sum by at most $O(J_m/\Lambda_m)=o(J_m)$, proving \eqref{eq:phase-block-cos-square}.

Finally, the bracketed error in \eqref{eq:pairwise-tail-coeff} is \(O(\Lambda_m^{-1})\) throughout the block. Using \((a+b)^2\ge a^2/2-b^2\), summing, and combining \eqref{eq:phase-block-cos-square} with \eqref{eq:block-envelope-comparable} gives
\begin{align*}
 \sum_{j=1}^{J_m}|a_{m+j}(\lambda_m)|^2
 &\ge cJ_m\lambda_m^{1/3}m^{-4/3}
   \ee^{2\mathcal A(m;\lambda_m)}.
\end{align*}
Since $J_m\asymp\delta_m^{-1}=(m/\lambda_m)^{1/3}$, this is exactly \eqref{eq:block-square-lower}.
\end{proof}

\subsection{From coefficient tails to approximation errors}
\label{subsec:tail-summation}
\label{subsec:error-main}

The phase-block lemma supplies the lower bound. For the upper bound, the monotone saddle action allows the entire infinite tail to be summed without coefficientwise monotonicity. The relevant width is again $\delta_m^{-1}$: for $k=m+j$ with $j=o(m)$, $k^{2/3}-m^{2/3}\sim\frac23m^{-1/3}j$, so \eqref{eq:action-decay-integrated} gives an envelope $\exp(-c\delta_mj)$. The integral comparison below is the global version of this picture.

\begin{lemma}[Tail summation bounds]\label{lem:tail-summation}
Under \eqref{eq:general-tail-sequence-assumption}, there is a constant $C>0$, depending only on $\alpha_0$, such that, for all sufficiently large $m$,
\begin{align}
 \sum_{k=m+1}^\infty |a_k(\lambda_m)|
 &\le C\lambda_m^{-1/6}m^{-1/3}
 \ee^{\mathcal A(m;\lambda_m)},
 \label{eq:l1-tail-bound}\\
 \sum_{k=m+1}^\infty |a_k(\lambda_m)|^2
 &\le C m^{-1}
 \ee^{2\mathcal A(m;\lambda_m)}.
 \label{eq:l2-tail-bound}
\end{align}
\end{lemma}

\begin{proof}
The uniform coefficient formula \eqref{eq:pairwise-tail-coeff}, together with \eqref{eq:B-comparable}, gives
\begin{equation}\label{eq:pointwise-envelope-upper}
 |a_k(\lambda_m)|
 \le C\lambda_m^{1/6}k^{-2/3}
 \ee^{\mathcal A(k;\lambda_m)},
 \qquad k\ge m+1,
\end{equation}
for all sufficiently large $m$.  Here the large-parameter condition is uniform because $(\lambda_m k^2)^{1/3}\ge\Lambda_m\to\infty$. By \eqref{eq:action-decay-integrated},
\begin{equation}\label{eq:action-envelope-from-m}
 \ee^{\mathcal A(k;\lambda_m)}
 \le
 \ee^{\mathcal A(m;\lambda_m)}
 \exp\!\left[-c\lambda_m^{1/3}
 \left(k^{2/3}-m^{2/3}\right)\right].
\end{equation}

Using integral comparison and the substitution $u=x^{2/3}$ gives
\begin{align*}
 &\int_m^\infty x^{-2/3}
 \exp\!\left[-c\lambda_m^{1/3}
 (x^{2/3}-m^{2/3})\right]dx \\
 &\qquad=\frac32\int_{m^{2/3}}^\infty
 u^{-1/2}\exp\!\left[-c\lambda_m^{1/3}(u-m^{2/3})\right]du
 \le C m^{-1/3}\lambda_m^{-1/3}.
\end{align*}
The first discrete term has the same order because $(\lambda_m/m)^{1/3}\le\alpha_0^{1/3}$.  Hence
\begin{align*}
 &\sum_{k=m+1}^\infty k^{-2/3}
 \exp\!\left[-c\lambda_m^{1/3}
 (k^{2/3}-m^{2/3})\right]\le
 C m^{-1/3}\lambda_m^{-1/3}.
\end{align*}
Multiplying by the prefactor $\lambda_m^{1/6}$ in \eqref{eq:pointwise-envelope-upper} proves \eqref{eq:l1-tail-bound}. The same substitution, now with the weight $x^{-4/3}$, gives
\begin{align*}
 &\sum_{k=m+1}^\infty k^{-4/3}
 \exp\!\left[-2c\lambda_m^{1/3}
 (k^{2/3}-m^{2/3})\right]\le
 C m^{-1}\lambda_m^{-1/3}.
\end{align*}
After multiplication by $\lambda_m^{1/3}$, this proves \eqref{eq:l2-tail-bound}.
\end{proof}

The phase block and the tail bounds now meet. Parseval converts the noncancelling block into a weighted-$L^2$ lower bound, while the squared-tail estimate gives the matching upper bound. The elementary inequality $\|g\|_{2,C}\le\sqrt\pi\|g\|_\infty$ then transfers the lower bound to best uniform approximation, and the absolute-tail estimate bounds the Chebyshev projection from above. The weighted norm retains the sharp algebraic scale; the only remaining algebraic gap comes from summing the uniform-norm tail without using cancellation.

The theorem is therefore intentionally sharper in the weighted $L^2$ norm than in the uniform norm.  Under the mild additional condition $\Lambda_m/\log m\to\infty$, the remaining algebraic factors are negligible relative to the saddle action; this holds for all three pole scalings below.

\begin{theorem}[Uniform approximation-error envelope]\label{thm:approx-error-envelope}
Fix $0<\alpha_0<\alpha_c$, let $\lambda_m>0$ satisfy \eqref{eq:general-tail-sequence-assumption}, and let $R_m(\lambda_m)=F_{\lambda_m}-\Pi_mF_{\lambda_m}$ denote the Chebyshev projection remainder. Then there exist constants $c,\,C>0$, depending only on $\alpha_0$, such that
\begin{equation}\label{eq:L2-error-two-sided}
 c\,m^{-1/2}
 H_e(\alpha_m)^m
 \le
 \|R_m(\lambda_m)\|_{2,C}
 \le
 C\,m^{-1/2}
 H_e(\alpha_m)^m
\end{equation}
for all sufficiently large $m$.  Furthermore,
\begin{equation}\label{eq:uniform-error-two-sided}
 c\,m^{-1/2}H_e(\alpha_m)^m
 \le E_m(\lambda_m)
 \le \|R_m(\lambda_m)\|_\infty
 \le C\lambda_m^{-1/6}m^{-1/3}H_e(\alpha_m)^m.
\end{equation}
If, in addition,
\begin{equation}\label{eq:log-equivalence-condition}
 \frac{\Lambda_m}{\log m}\longrightarrow\infty,
\end{equation}
then the logarithms of the weighted projection, uniform projection, and best uniform approximation errors are all asymptotic to the saddle action:
\begin{equation}\label{eq:log-equivalence-errors}
 \log\|R_m(\lambda_m)\|_{2,C}
 \sim
 \log\|R_m(\lambda_m)\|_\infty
 \sim
 \log E_m(\lambda_m)
 \sim
 \mathcal A(m;\lambda_m).
\end{equation}
\end{theorem}

\begin{proof}
By \eqref{eq:parseval-tail}, Lemma~\ref{lem:phase-block} gives
\[
 \|R_m(\lambda_m)\|_{2,C}^2
 \ge \frac{\pi}{2}
 \sum_{j=1}^{J_m}|a_{m+j}(\lambda_m)|^2
 \ge c m^{-1}\ee^{2\mathcal A(m;\lambda_m)}.
\]
The upper bound in \eqref{eq:L2-error-two-sided} follows from \eqref{eq:parseval-tail} and \eqref{eq:l2-tail-bound}.  Since $\mathcal A(m;\lambda_m)=m\log H_e(\lambda_m/m)=m\log H_e(\alpha_m)$, this proves \eqref{eq:L2-error-two-sided}.

For any $p\in\Pp_m$, orthogonality gives $\|F_{\lambda_m}-p\|_{2,C}\ge\|F_{\lambda_m}-\Pi_mF_{\lambda_m}\|_{2,C}$, while $\|g\|_{2,C}\le\sqrt\pi\|g\|_\infty$. Taking the infimum over $p$ therefore yields $E_m(\lambda_m)\ge\pi^{-1/2}\|F_{\lambda_m}-\Pi_mF_{\lambda_m}\|_{2,C}$, the first inequality in \eqref{eq:uniform-error-two-sided}. The middle inequality is immediate from the definition of $E_m$.  Finally, \eqref{eq:sup-tail-upper-basic} and Lemma~\ref{lem:tail-summation} give the last inequality.

It remains to justify the logarithmic conclusion under \eqref{eq:log-equivalence-condition}. By Lemma~\ref{lem:small-alpha-action} and continuity on $[0,\alpha_0]$, $-\log H_e(\alpha)/\alpha^{1/3}$ extends continuously to $3/2$ at $\alpha=0$ and is bounded above and below by positive constants. Hence $-\mathcal A(m;\lambda_m)=-m\log H_e(\alpha_m)\asymp m\alpha_m^{1/3}=\Lambda_m$.
Moreover, since $\Lambda_m\to\infty$ and $\alpha_m\le\alpha_0$, eventually $m^{-2}\le\lambda_m=\Lambda_m^3/m^2\le\alpha_0m$, hence $|\log\lambda_m|=O(\log m)$.  Taking logarithms in \eqref{eq:L2-error-two-sided}--\eqref{eq:uniform-error-two-sided} therefore changes $\mathcal A(m;\lambda_m)$ by at most $O(\log m)$.  Under \eqref{eq:log-equivalence-condition} this is $o(\Lambda_m)=o(|\mathcal A(m;\lambda_m)|)$, which proves \eqref{eq:log-equivalence-errors}.
\end{proof}

\begin{remark}[The origin of the $m^{-1/2}$ weighted scale]\label{rem:L2-prefactor-mechanism}
The two-sided power $m^{-1/2}$ in \eqref{eq:L2-error-two-sided} has a simple structural origin. Near index $m$, one coefficient and a noncancelling block have scales
\[
 \lambda_m^{1/6}m^{-2/3}\ee^{\mathcal A(m;\lambda_m)},
 \qquad J_m\asymp(m/\lambda_m)^{1/3},
\]
respectively. Hence $\lambda_m^{1/6}m^{-2/3}J_m^{1/2}\asymp m^{-1/2}$. The same weighted-$L^2$ scale persists from fixed to linearly moving poles although the effective tail width changes from order $m^{1/3}$ to order one.
\end{remark}

\begin{remark}[The remaining uniform-norm gap]\label{rem:uniform-prefactor-scope}
For fixed $\lambda$, the uniform-norm bounds in \eqref{eq:uniform-error-two-sided} have prefactors $m^{-1/2}$ and $m^{-1/3}$. Closing this gap requires a sharper analysis of the discrete oscillatory tail. It does not affect the logarithmic laws; for $\lambda_m\asymp m$, the two powers already coincide.
\end{remark}

The coefficient-envelope transition now transfers directly to the approximation errors.

\begin{corollary}[Approximation-error transition]\label{cor:error-transition}
The following limits hold for each of the three errors $D_m\in\{E_m(\lambda_m),\|R_m(\lambda_m)\|_\infty,\|R_m(\lambda_m)\|_{2,C}\}$.
\begin{enumerate}
\item[(i)] If
\begin{equation}\label{eq:error-sublinear-assumption}
 \lambda_m\sim c m^\gamma,
 \qquad c>0,\qquad 0\le\gamma<1,
\end{equation}
then
\begin{equation}\label{eq:error-sublinear-log-rate}
 \lim_{m\to\infty}\frac{\log D_m}{m^{(2+\gamma)/3}}=-\frac32c^{1/3}.
\end{equation}

For fixed $\lambda>0$, the common logarithmic scale is $-\tfrac32(\lambda m^2)^{1/3}$.

\item[(ii)] If
\begin{equation}\label{eq:error-linear-assumption}
 \frac{\lambda_m}{m}\longrightarrow\alpha
 \in(0,\alpha_c),
\end{equation}
then
\begin{equation}\label{eq:error-linear-root-rate}
 \lim_{m\to\infty}D_m^{1/m}=H_e(\alpha).
\end{equation}

If $\lambda_m=\alpha m$ exactly, then all three errors are bounded above and below by positive constants times
\begin{equation}\label{eq:linear-error-order}
 m^{-1/2}H_e(\alpha)^m.
\end{equation}
\end{enumerate}
\end{corollary}

\begin{proof}
Under \eqref{eq:error-sublinear-assumption}, Theorem~\ref{thm:pole-scaling-transition} gives $m\log H_e(\lambda_m/m)=-\frac32c^{1/3}m^{(2+\gamma)/3}+o(m^{(2+\gamma)/3})$. All algebraic factors in \eqref{eq:L2-error-two-sided}--\eqref{eq:uniform-error-two-sided} contribute only $O(\log m)$ to the logarithm, proving \eqref{eq:error-sublinear-log-rate}.  Under \eqref{eq:error-linear-assumption}, taking $m$th roots in the same bounds gives \eqref{eq:error-linear-root-rate}.  If $\lambda_m=\alpha m$, the last factor in \eqref{eq:uniform-error-two-sided} is $\lambda_m^{-1/6}m^{-1/3}=\alpha^{-1/6}m^{-1/2}$, so the upper and lower algebraic orders coincide.
\end{proof}

\subsection{An explicit oscillatory profile in the linear regime}
\label{subsec:l2-profile}

When the pole-to-degree ratio is fixed exactly, the relevant tail has bounded width rather than a growing width. Consequently, the first few omitted modes do not average out: their phase survives in the normalized weighted-$L^2$ error. The next result resolves this residual oscillation explicitly, sharpening the two-sided order in Theorem~\ref{thm:approx-error-envelope}.

\begin{proposition}[Oscillatory weighted-$L^2$ profile]\label{prop:l2-oscillatory-profile}
Let $K\Subset(0,\alpha_c)$ be compact. For $\alpha\in K$ set $\lambda_m=\alpha m$ and write
\begin{equation}\label{eq:linear-profile-phase}
 \varphi_m(\alpha):=m\Theta(\alpha)+\beta(\alpha),
 \qquad
 \mathfrak r(\alpha)=\rho(\alpha)\ee^{i\omega(\alpha)}
\end{equation}
as in \eqref{eq:R-rho-omega-def}. Define the positive, $\pi$-periodic profile
\begin{equation}\label{eq:profile-function-series}
 \mathcal P_\alpha(\varphi)
 :=\left[\textstyle 
 \sum_{j=1}^{\infty}\rho(\alpha)^{2j}
 \cos^2\!\bigl(\varphi+j\omega(\alpha)\bigr)
 \right]^{1/2}.
\end{equation}
Then, uniformly for $\alpha\in K$,
\begin{equation}\label{eq:L2-oscillatory-profile}
 \frac{\sqrt m}{H_e(\alpha)^m}
 \|R_m(\alpha m)\|_{2,C}
 =\sqrt{\frac{\pi}{2}}\,B(\alpha)
 \mathcal P_\alpha\!\bigl(\varphi_m(\alpha)\bigr)
 +O_K(m^{-1}).
\end{equation}
The profile has the closed form
\begin{align}\label{eq:profile-function-closed}
 \mathcal P_\alpha(\varphi)^2
 ={}&\frac{\rho(\alpha)^2}{2[1-\rho(\alpha)^2]}+\frac12\Rea\!\left
 \{\frac{\rho(\alpha)^2
 \exp\bigl(2i[\varphi+\omega(\alpha)]\bigr)}
 {1-\rho(\alpha)^2\exp(2i\omega(\alpha))}\right\}.
\end{align}
Moreover, on every such compact $K$, $\mathcal P_\alpha(\varphi)$ is bounded above and away from zero uniformly in $(\alpha,\varphi)\in K\times\R$. Thus \eqref{eq:L2-oscillatory-profile} refines the order $m^{-1/2}H_e(\alpha)^m$ through an explicit oscillatory profile.
\end{proposition}

\begin{proof}
Fix a compact interval $K\Subset(0,\alpha_c)$ and choose $\eta>0$ so small that
\[
 K_1:=\left\{\frac{\alpha}{1+s}:\alpha\in K,\ 0\le s\le\eta\right\}
 \Subset(0,\alpha_c).
\]
For $1\le j\le\eta m$, the coefficient index $k=m+j$ therefore has $\alpha m/k\in K_1$. The additive error in the brackets of \eqref{eq:pairwise-tail-coeff} is $O_K(m^{-1})$ on this range. In addition, Lemma~\ref{lem:tail-action-derivative} and smoothness on $K_1$ give
\[
 \mathcal S(m+j;\alpha m)
 =m\psi_\alpha(\zeta_\alpha)
  +j\log\mathfrak r(\alpha)+O_K(j^2/m),
\]
and
\[
 (m+j)^{-1/2}B\!\left(\frac{\alpha m}{m+j}\right)
 =m^{-1/2}B(\alpha)\left(1+O_K(j/m)\right).
\]
For $1\le j\le m^{1/3}$, the same Taylor expansion applied to the phase shift $\beta$ leads to
\begin{align}\label{eq:linear-tail-mode-profile}
 &\frac{\sqrt m}{H_e(\alpha)^m}(-1)^{m+j}a_{m+j}(\alpha m)\nonumber\\
 &\quad=B(\alpha)\rho(\alpha)^j
 \cos\!\bigl(\varphi_m(\alpha)+j\omega(\alpha)\bigr)
 +O_K\!\left(\frac{(1+j)^2}{m}\,\ee^{-c_Kj}\right),
\end{align}
where $c_K>0$ is independent of $m$, $j$, and $\alpha\in K$. Here we used that $\rho(\alpha)$ is uniformly bounded below one on $K$.

It remains to justify summation over the full tail. For $m^{1/3}<j\le\eta m$, the real action derivative $\log\rho(\alpha m/x)$ is bounded above by a negative constant uniformly for $\alpha\in K$, so the normalized coefficients are $O_K(\ee^{-c_Kj})$. For $k>(1+\eta)m$, the integrated action bound \eqref{eq:action-decay-integrated}, together with the positive lower bound for $\alpha$ on $K$, gives an $O_K(\ee^{-c_Km})$ contribution to the squared tail. Hence Parseval's identity and \eqref{eq:linear-tail-mode-profile} imply
\begin{align*}
 &\frac{m}{H_e(\alpha)^{2m}}\|R_m(\alpha m)\|_{2,C}^2=\frac{\pi}{2}B(\alpha)^2
 \sum_{j=1}^{\infty}\rho(\alpha)^{2j}
 \cos^2\!\bigl(\varphi_m(\alpha)+j\omega(\alpha)\bigr)
 +O_K(m^{-1})
\end{align*}
uniformly for $\alpha\in K$.

The geometric-series identity $\cos^2y=(1+\cos2y)/2$ gives \eqref{eq:profile-function-closed}. The series in \eqref{eq:profile-function-series} never vanishes: otherwise all phases $\varphi+j\omega(\alpha)$ would be congruent to $\pi/2$ modulo $\pi$, forcing $\omega(\alpha)\equiv0\pmod\pi$, contrary to $0<\omega(\alpha)<\pi$. Continuity on the compact set $K\times[0,\pi]$ therefore gives a uniform positive lower bound for $\mathcal P_\alpha$. Since $B$ is also bounded above and away from zero on $K$, taking square roots in the preceding squared-norm expansion proves \eqref{eq:L2-oscillatory-profile} with the stated $O_K(m^{-1})$ remainder.
\end{proof}

Andersson's lower-bound proof invokes standard saddle-point methods only to establish the $m$th-root lower bound, and his final explicit optimal construction gives a constant-multiple geometric upper bound without an algebraic prefactor or oscillatory profile \cite[Secs.~6 and~8]{Andersson1981}. The profile above resolves both finer scales. For $\alpha_m\downarrow0$, the growing block in Lemma~\ref{lem:phase-block}, rather than a fixed number of coefficients, supplies uniform noncancellation.

\paragraph{Returning to the half-line.}
The isometry $\mathcal E_m^{\rm rat}(t,q)=E_m(tq)$ from Proposition~\ref{prop:mobius-isometry} transfers these results to repeated-pole rational approximation.

\begin{corollary}[Restricted-denominator error laws]\label{cor:rational-error-laws}
For $t>0$, all conclusions of Corollary~\ref{cor:error-transition} hold for $\mathcal E_m^{\rm rat}(t,q_m)=E_m(tq_m)$ after setting $\lambda_m=tq_m$. In particular,
\begin{equation}\label{eq:rational-sublinear-error-law}
 q_m\sim cm^\gamma,\quad c>0,\quad0\le\gamma<1
 \quad\Longrightarrow\quad
 \frac{\log\mathcal E_m^{\rm rat}(t,q_m)}{m^{(2+\gamma)/3}}
 \longrightarrow-\frac32(tc)^{1/3}.
\end{equation}
If $tq_m/m\to\alpha\in(0,\alpha_c)$, then
\begin{equation}\label{eq:rational-linear-error-law}
 \left(\mathcal E_m^{\rm rat}(t,q_m)\right)^{1/m}\longrightarrow H_e(\alpha).
\end{equation}
For the exact family $q_m=(\alpha/t)m$, the stronger conclusion is
\begin{equation}\label{eq:rational-linear-two-sided-order}
 \mathcal E_m^{\rm rat}(t,q_m)\asymp m^{-1/2}H_e(\alpha)^m.
\end{equation}
At $q_m=m/(\sqrt2\,t)$ this becomes
\begin{equation}\label{eq:rational-optimal-strong-order}
 \mathcal E_m^{\rm rat}(t,q_m)\asymp m^{-1/2}(\sqrt2-1)^m.
\end{equation}
\end{corollary}

\begin{proof}
Use the isometry \eqref{eq:rational-polynomial-isometry-front} and the value $H_e(\alpha_*)=\sqrt2-1$.
\end{proof}

Thus the classical root-rate statement can be sharpened without changing its rate function. The Gaussian saddle width produces the two-sided factor $m^{-1/2}$, while the exponential factor must remain $H_e(tq_m/m)^m$ unless the pole-to-degree ratio is fixed exactly.

\begin{remark}[Transfer to stable polynomial schemes]\label{rem:stable-transfer}
Suppose a polynomial process $\mathcal P_mF_{\lambda_m}\in\Pp_m$ satisfies
\begin{equation}\label{eq:quasibest-process}
 E_m(\lambda_m)\le
 \|F_{\lambda_m}-\mathcal P_mF_{\lambda_m}\|_\infty
 \le C_mE_m(\lambda_m),\qquad C_m\ge1.
\end{equation}
If $\log C_m=o(m^{(2+\gamma)/3})$ in the sublinear regime, or $o(m)$ in the linear regime, the corresponding logarithmic or root rate is unchanged. Chebyshev interpolation is one example, since its Lebesgue constant grows only logarithmically.
\end{remark}

\section{Precision-to-work laws and robust pole design}
\label{sec:design}

We now turn the asymptotic error law into a work law. First we invert the fixed-ratio estimate to choose a degree and pole for a prescribed tolerance. We then address a practical mismatch: in an iterative matrix-function computation the pole may be chosen from a nominal degree before the actual stopping degree is known. The ratio between those two degrees changes the effective approximation parameter.

Throughout this section, set
\begin{equation}\label{eq:kappa-def}
 \kappa(\alpha):=-\log H_e(\alpha)>0,
 \qquad 0<\alpha<\alpha_c,
\end{equation}
and recall from Proposition~\ref{prop:He-derivatives} that
\begin{equation}\label{eq:optimal-kappa-data}
 \alpha_*:=\frac{1}{\sqrt2},\;\,
 \rho_*:=H_e(\alpha_*)=\sqrt2-1,
 \;\,
 \kappa_*:=\kappa(\alpha_*)
 =\log(1+\sqrt2)=\operatorname{arsinh}(1).
\end{equation}
Here, \(\kappa(\alpha)\) is the asymptotic error reduction per degree.

\subsection{From a prescribed tolerance to degree and pole}
\label{subsec:tolerance-design}

Fix \(t>0\) and a linear pole ratio \(\alpha\in(0,\alpha_c)\). For every degree \(m\), choose
\begin{equation}\label{eq:qmalpha}
 q_m(\alpha)=\frac{\alpha m}{t},
\end{equation}
so that the transformed parameter satisfies \(tq_m(\alpha)=\alpha m\). Define the first degree in this family that meets a uniform tolerance \(\varepsilon\) by
\begin{equation}\label{eq:meps-alpha}
 m_\varepsilon(\alpha)
 :=\min\left\{m\in\mathbb N:
 \mathcal E_m^{\rm rat}(t,q_m(\alpha))\le\varepsilon\right\}.
\end{equation}
For sufficiently small \(\varepsilon\), the set in \eqref{eq:meps-alpha} is nonempty by Corollary~\ref{cor:rational-error-laws}.

\begin{theorem}[Precision-to-work law at a fixed pole ratio]\label{thm:precision-work-fixed-alpha}
Fix a compact interval \(K\Subset(0,\alpha_c)\). Uniformly for \(\alpha\in K\), as \(\varepsilon\downarrow0\),
\begin{equation}\label{eq:precision-work-fixed-alpha}
 m_\varepsilon(\alpha)
 =\frac{L_\varepsilon}{\kappa(\alpha)}
 -\frac{\log L_\varepsilon}{2\kappa(\alpha)}
 +O(1),
 \qquad
 L_\varepsilon:=\log\frac1\varepsilon.
\end{equation}
In particular,
\begin{equation}\label{eq:precision-work-leading}
 \lim_{\varepsilon\downarrow0}
 \frac{m_\varepsilon(\alpha)}{\log(1/\varepsilon)}
 =\frac1{\kappa(\alpha)}.
\end{equation}
\end{theorem}

\begin{proof}
There are two points to check.  First, the first-passage degree must tend to infinity uniformly for $\alpha\in K$, so that the two-sided asymptotic bounds are available at the crossing.  Second, once this is known, the problem reduces to inverting the model scale $m^{-1/2}e^{-\kappa(\alpha)m}$.

For each fixed $m$, the map $\alpha\mapsto\mathcal E_m^{\rm rat}(t,q_m(\alpha))=E_m(\alpha m)$ is continuous and strictly positive on $K$: continuity follows from uniform continuity of $F_{\alpha m}$ in $\alpha$, and positivity because $F_{\alpha m}$ is not a polynomial of degree at most $m$. Hence, for each fixed $M$, $d_{K,M}:=\min_{\alpha\in K,\,1\le m\le M}\mathcal E_m^{\rm rat}(t,q_m(\alpha))>0$, so $\inf_{\alpha\in K}m_\varepsilon(\alpha)\to\infty$ as $\varepsilon\downarrow0$. Thus the first-passage degree eventually lies in the range where the linear-regime bounds are valid uniformly on $K$.

Indeed, Corollary~\ref{cor:rational-error-laws} and Theorem~\ref{thm:approx-error-envelope}, uniformly on compact subcritical intervals, give constants \(0<c_K\le C_K<\infty\) and \(M_0\) such that
\begin{equation}\label{eq:linear-family-two-sided}
 c_Km^{-1/2}\ee^{-\kappa(\alpha)m}
 \le \mathcal E_m^{\rm rat}(t,q_m(\alpha))
 \le C_Km^{-1/2}\ee^{-\kappa(\alpha)m}
\end{equation}
for all \(m\ge M_0\) and \(\alpha\in K\). Since \(\kappa\) is bounded above and away from zero on \(K\), set \(F_\alpha(x):=\kappa(\alpha)x+\frac12\log x\) for \(x\ge1\); it is strictly increasing. Since $m_\varepsilon(\alpha)\to\infty$ uniformly, for small $\varepsilon$ also $m_\varepsilon(\alpha)-1\ge M_0$. The lower bound at $m=m_\varepsilon(\alpha)$ gives $F_\alpha(m_\varepsilon(\alpha))\ge L_\varepsilon+\log c_K$, while minimality and the upper bound at $m=m_\varepsilon(\alpha)-1$ give $F_\alpha(m_\varepsilon(\alpha)-1)<L_\varepsilon+\log C_K$. Since $F_\alpha(m)-F_\alpha(m-1)=O(1)$ uniformly on $K$, we obtain
\[
 F_\alpha(m_\varepsilon(\alpha))=L_\varepsilon+O(1).
\]
The real equation $\kappa(\alpha)x+\frac12\log x=L_\varepsilon+O(1)$ has $x\asymp L_\varepsilon$ uniformly on $K$. Substituting $\log x=\log L_\varepsilon+O(1)$ back gives $x=L_\varepsilon/\kappa(\alpha)-\log L_\varepsilon/(2\kappa(\alpha))+O(1)$ uniformly on $K$, which proves \eqref{eq:precision-work-fixed-alpha}. Dividing by $L_\varepsilon$ gives \eqref{eq:precision-work-leading}.
\end{proof}

The leading work is minimized by maximizing \(\kappa(\alpha)\), or equivalently by minimizing \(H_e(\alpha)\). Proposition~\ref{prop:He-derivatives} hence selects the classical ratio \(\alpha_*=1/\sqrt2\).

\begin{corollary}[Asymptotically optimal tolerance-driven design]\label{cor:optimal-tolerance-design}
Choose the explicit pole family
\begin{equation}\label{eq:optimal-pole-family-design}
 q_m^*=\frac{m}{\sqrt2\,t}.
\end{equation}
Let \(m_\varepsilon^*\) be the first degree in this family for which the best restricted-denominator error is at most \(\varepsilon\). Then
\begin{equation}\label{eq:optimal-meps}
 m_\varepsilon^*
 =\frac{L_\varepsilon}{\operatorname{arsinh}(1)}
 -\frac{\log L_\varepsilon}{2\operatorname{arsinh}(1)}
 +O(1),
\end{equation}
and the corresponding pole satisfies
\begin{equation}\label{eq:optimal-qeps}
 q_\varepsilon^*:=\frac{m_\varepsilon^*}{\sqrt2\,t}
 =\frac{1}{\sqrt2\,t\operatorname{arsinh}(1)}
 \left(L_\varepsilon-\frac12\log L_\varepsilon\right)
 +O(t^{-1}).
\end{equation}
Thus the leading number of degrees needed for a prescribed half-line uniform tolerance is independent of \(t\); the time scale is absorbed by the pole location. Since $\log 10/\operatorname{arsinh}(1)\approx2.6125$, each additional decimal digit of uniform accuracy costs about \(2.61\) degrees at leading order.
\end{corollary}

\begin{proof}
Apply Theorem~\ref{thm:precision-work-fixed-alpha} at \(\alpha=\alpha_*\) and use \eqref{eq:optimal-kappa-data}. Formula \eqref{eq:optimal-qeps} follows from \(q_\varepsilon^*=m_\varepsilon^*/(\sqrt2 t)\).
\end{proof}

The two-term law above concerns the explicit family \eqref{eq:optimal-pole-family-design}. It is useful to separate this statement from optimality over \emph{all} repeated-pole choices. Define $\mathcal E_m^{\rm conc}(t):=\inf_{q>0}\mathcal E_m^{\rm rat}(t,q)$ and $M_\varepsilon:=\min\{m:\mathcal E_m^{\rm conc}(t)\le\varepsilon\}$.

Andersson's theorem applies to arbitrary positive sequences of concentrated-pole ratios and identifies the corresponding limsup and liminf root errors with the saddle factors \cite[Sec.~2, p.~86]{Andersson1981}. If $\mathcal E_m^{\rm conc}(t)$ had a subsequence with root factor strictly below $\sqrt2-1$, we could select poles realizing that gap on the subsequence and choose the remaining poles arbitrarily. Andersson's lower-root bound would then give a contradiction. The explicit choice $q_m=m/(\sqrt2 t)$ supplies the matching upper bound; hence $(\mathcal E_m^{\rm conc}(t))^{1/m}\to\sqrt2-1$. Since $\mathcal E_m^{\rm conc}(t)$ is nonincreasing in $m$, inversion of this root limit yields
\[
 \lim_{\varepsilon\downarrow0}\frac{M_\varepsilon}{\log(1/\varepsilon)}
 =\frac1{\operatorname{arsinh}(1)}.
\]
Thus the explicit family \eqref{eq:optimal-pole-family-design} is optimal in leading degree complexity within the concentrated-pole class; the second term in \eqref{eq:optimal-meps} is the stronger conclusion available for that explicit family.

\begin{remark}[Finite-tolerance use]\label{rem:tolerance-certification}
The $O(1)$ term in the work law does not give a finite-tolerance certificate. For such a guarantee, evaluate the scalar error, or a computable bound, at integer degrees near the prediction before fixing the matrix pole. The work variable counts shifted solves; setup, orthogonalization, and inexact-solve costs require an implementation-specific model.
\end{remark}

\subsection{Stopping mismatch and robust pole design}
\label{subsec:mismatch}
\label{subsec:robust-design}

The preceding design assumes that the pole is chosen for the degree eventually used. An iterative matrix-function algorithm often proceeds in the opposite order: the pole is fixed before iteration, whereas the stopping degree is decided later from a residual or error criterion. We therefore distinguish the nominal degree \(N\), used only to select the pole, from the realized degree \(m_N\) at which the computation stops.

Let
\begin{equation}\label{eq:nominal-pole}
 q_N=\frac{\widehat\alpha N}{t},
 \qquad \widehat\alpha>0,
\end{equation}
and suppose
\begin{equation}\label{eq:realized-ratio}
 \frac{m_N}{N}\longrightarrow\vartheta>0.
\end{equation}
The transformed parameter at the realized degree is then
\begin{equation}\label{eq:effective-alpha}
 \frac{tq_N}{m_N}
 =\widehat\alpha\frac{N}{m_N}
 \longrightarrow
 \alpha_{\rm eff}:=\frac{\widehat\alpha}{\vartheta}.
\end{equation}

\begin{theorem}[Degree--pole mismatch law]\label{thm:mismatch}
Assume
\begin{equation}\label{eq:mismatch-subcritical}
 0<\alpha_{\rm eff}=\frac{\widehat\alpha}{\vartheta}<\alpha_c.
\end{equation}
Then
\begin{equation}\label{eq:mismatch-root}
 \lim_{N\to\infty}
 \left(\mathcal E_{m_N}^{\rm rat}(t,q_N)\right)^{1/m_N}
 =H_e\!\left(\frac{\widehat\alpha}{\vartheta}\right).
\end{equation}
After the M\"obius transformation, the same root factor is inherited by every polynomial approximation process satisfying the subexponential quasi-best condition of Remark~\ref{rem:stable-transfer}.
\end{theorem}

\begin{proof}
At degree \(m_N\), the transformed parameter is \(\lambda_N=tq_N\), and \eqref{eq:effective-alpha} gives \(\lambda_N/m_N\to\alpha_{\rm eff}\). Equation \eqref{eq:mismatch-root} therefore follows directly from Corollary~\ref{cor:rational-error-laws}. The final assertion follows from Remark~\ref{rem:stable-transfer}.
\end{proof}

Two consequences make the mismatch law useful in practice. First, a sublinear discrepancy between nominal and realized degree does not change the optimal root factor. Second, even a small fixed relative discrepancy has only a quadratic effect when the pole was designed at the optimum.

\begin{corollary}[Stopping robustness at the classical optimum]\label{cor:stopping-robustness}
Let \(\widehat\alpha=\alpha_*=1/\sqrt2\) in \eqref{eq:nominal-pole}.
\begin{enumerate}
\item[(i)] If \(m_N=N+o(N)\), then
\begin{equation}\label{eq:sublinear-mismatch-optimal}
 \lim_{N\to\infty}
 \left(\mathcal E_{m_N}^{\rm rat}(t,q_N)\right)^{1/m_N}
 =\sqrt2-1.
\end{equation}
\item[(ii)] If \(m_N/N\to1+\epsilon\) with \(|\epsilon|\) sufficiently small, then the realized root factor satisfies
\begin{equation}\label{eq:quadratic-mismatch}
 \log H_e\!\left(\frac{\alpha_*}{1+\epsilon}\right)
 =\log(\sqrt2-1)+\frac{\epsilon^2}{5\sqrt2}+O(\epsilon^3).
\end{equation}
Thus there is no first-order loss in the asymptotic convergence factor under a small relative degree mismatch.
\end{enumerate}
\end{corollary}

\begin{proof}
Part (i) follows from Theorem~\ref{thm:mismatch} because \(N/m_N\to1\). For part (ii), set \(h(\alpha)=\log H_e(\alpha)\). Since \(h'(\alpha_*)=0\) and \(h''(\alpha_*)=2\sqrt2/5\) by Proposition~\ref{prop:He-derivatives}, while
\[
 \frac{\alpha_*}{1+\epsilon}-\alpha_*
 =-\alpha_*\epsilon+O(\epsilon^2),
\]
Taylor expansion gives \eqref{eq:quadratic-mismatch} because \(\alpha_*^2=1/2\).
\end{proof}

\paragraph{Robust design under uncertain stopping.}
The mismatch law also suggests a natural robust-design question. Suppose the stopping degree is not known when the pole is chosen, but prior information places the realized-to-nominal ratio in an interval,
\begin{equation}\label{eq:theta-uncertainty}
 0<\vartheta_-<\vartheta_+<\infty,
 \qquad
 \frac{m_N}{N}\in[\vartheta_-,\vartheta_+]
\end{equation}
asymptotically. A natural criterion is to minimize the worst asymptotic error reduction factor per realized degree,
\begin{equation}\label{eq:robust-objective}
 \mathfrak H(\widehat\alpha)
 :=\max_{\vartheta\in[\vartheta_-,\vartheta_+]}
 H_e\!\left(\frac{\widehat\alpha}{\vartheta}\right).
\end{equation}
This is a root-factor robustness criterion. It should not be confused with minimizing the absolute error for a fixed nominal budget \(N\), which would weight the logarithmic factor by the realized ratio \(\vartheta\) and leads to a different optimization problem. Under the change of variable $\tau=1/\vartheta$, the scalar minimax structure in \eqref{eq:robust-objective} is the same endpoint-equalization problem that appears in recent time-interval pole design \cite{GuettelShao2025,GuettelShao2026}. We do not claim the equalization principle as new; its role here is to translate uncertainty in the realized stopping degree into an effective pole ratio. The local mismatch and robust-ratio expansions below are specific to that interpretation.

\begin{theorem}[Robust pole ratio under stopping uncertainty]\label{thm:robust-pole}
Assume that there exists \(\alpha_0<\alpha_c\) such that
\begin{equation}\label{eq:robust-subcritical-condition}
 \alpha_*\frac{\vartheta_+}{\vartheta_-}<\alpha_0.
\end{equation}
Then the minimization of \eqref{eq:robust-objective} over
\begin{equation}\label{eq:robust-admissible-domain}
 0<\widehat\alpha\le\alpha_0\vartheta_-
\end{equation}
has a unique solution
\begin{equation}\label{eq:robust-alpha-interval}
 \widehat\alpha_{\rm rob}
 \in(\alpha_*\vartheta_-,\alpha_*\vartheta_+).
\end{equation}
It is characterized by the endpoint equalization condition
\begin{equation}\label{eq:robust-equalization}
 H_e\!\left(\frac{\widehat\alpha_{\rm rob}}{\vartheta_-}\right)
 =
 H_e\!\left(\frac{\widehat\alpha_{\rm rob}}{\vartheta_+}\right).
\end{equation}
The corresponding preselected physical pole is
\begin{equation}\label{eq:robust-physical-pole}
 q_N^{\rm rob}=\frac{\widehat\alpha_{\rm rob}N}{t}.
\end{equation}
\end{theorem}

\begin{proof}
If \(\widehat\alpha<\alpha_*\vartheta_-\), then every effective ratio \(\widehat\alpha/\vartheta\) lies to the left of \(\alpha_*\), where \(H_e\) is strictly decreasing. Increasing \(\widehat\alpha\) therefore decreases \eqref{eq:robust-objective}. Similarly, if \(\widehat\alpha>\alpha_*\vartheta_+\), then every effective ratio lies to the right of \(\alpha_*\), and decreasing \(\widehat\alpha\) improves the objective. Hence any minimizer must lie in the interval in \eqref{eq:robust-alpha-interval}.

For such a ratio, the effective interval $[\widehat\alpha/\vartheta_+,\widehat\alpha/\vartheta_-]$ contains \(\alpha_*\). Since \(H_e\) decreases to \(\alpha_*\) and increases thereafter, the maximum is attained at an endpoint. Define $G(\widehat\alpha):=\log H_e(\widehat\alpha/\vartheta_-)-\log H_e(\widehat\alpha/\vartheta_+)$. At \(\widehat\alpha=\alpha_*\vartheta_-\) one has \(G<0\), whereas at \(\widehat\alpha=\alpha_*\vartheta_+\) one has \(G>0\).

Throughout the interval in \eqref{eq:robust-alpha-interval}, the first argument of \(H_e\) lies to the right of \(\alpha_*\) and the second to the left. Hence $G'(\widehat\alpha)=h'(\widehat\alpha/\vartheta_-)/\vartheta_- -h'(\widehat\alpha/\vartheta_+)/\vartheta_+>0$, where $h=\log H_e$. Thus $G$ has exactly one zero. There the two endpoint errors are equal; moving the pole ratio in either direction makes one endpoint strictly worse. This proves the minimax characterization. Condition \eqref{eq:robust-subcritical-condition} keeps the whole candidate interval inside the subcritical theory.
\end{proof}

For a narrow uncertainty interval the robust ratio is close to the nominal optimum, but the first nontrivial correction is not obtained by simply centering the two degree ratios arithmetically.

\begin{corollary}[Narrow symmetric stopping interval]\label{cor:narrow-robust}
Let
\begin{equation}\label{eq:symmetric-theta-interval}
 \vartheta_- =\vartheta_0(1-\epsilon),
 \qquad
 \vartheta_+ =\vartheta_0(1+\epsilon),
 \qquad \epsilon\downarrow0.
\end{equation}
Then the robust pole ratio of Theorem~\ref{thm:robust-pole} satisfies
\begin{equation}\label{eq:robust-alpha-expansion}
 \widehat\alpha_{\rm rob}
 =\alpha_*\vartheta_0
 \left(1-\frac{107}{150}\epsilon^2+O(\epsilon^4)\right),
\end{equation}
and the worst root factor obeys
\begin{equation}\label{eq:robust-factor-expansion}
 \log\mathfrak H(\widehat\alpha_{\rm rob})
 =\log(\sqrt2-1)+\frac{\epsilon^2}{5\sqrt2}+O(\epsilon^4).
\end{equation}
\end{corollary}

The expansion follows by applying the implicit-function theorem to the endpoint-equalization equation and expanding through cubic order in the stopping uncertainty. The derivative calculation is recorded in Appendix~\ref{app:robust-expansion}.

The robust objective is the asymptotic root factor \emph{per realized degree}. Matrix-dependent stopping tests can exploit more information than this scalar worst-case criterion, as discussed next.

\section{Matrix implications}
\label{sec:matrix}

The scalar theory now returns to the matrix problem that motivated it. For self-adjoint negative semidefinite matrices, the spectral theorem transfers half-line approximation bounds directly to matrix actions. The repeated pole has the complementary computational benefit that the same shifted factorization can be reused. We record first the exact worst-case interpretation and then the shift-and-invert Krylov bound.

\subsection{Scalar error and the matrix worst case}
\label{subsec:matrix-worstcase}

Let $e_r(z)=\ee^{tz}-r(z)$, where $r$ has no pole on $(-\infty,0]$.
\begin{proposition}[Spectral measure and worst-case equality]\label{prop:matrix-worstcase}
For $A=A^*\le0$, let $\mu_b(\Omega)=\langle E_A(\Omega)b,b\rangle$, where $E_A$ is its spectral resolution. Then
\begin{equation}\label{eq:spectral-measure-error}
 \|\ee^{tA}b-r(A)b\|_2^2
 =\int_{\sigma(A)}|e_r(\lambda)|^2\,d\mu_b(\lambda).
\end{equation}
Moreover, for every fixed dimension $N\ge1$,
\begin{equation}\label{eq:exact-worstcase-equality}
 \sup_{\substack{A\in\C^{N\times N},\ A=A^*\le0\\\|b\|_2=1}}
 \|\ee^{tA}b-r(A)b\|_2=\sup_{z\le0}|e_r(z)|.
\end{equation}
Taking the infimum over $r\in\mathcal R_m(q)$ gives $\mathcal E_m^{\rm rat}(t,q)$ as the exact minimax error when one rational function is chosen before the matrix and vector.
\end{proposition}

\begin{proof}
The spectral theorem gives \eqref{eq:spectral-measure-error} and $\|\ee^{tA}-r(A)\|_2=\max_{\lambda\in\sigma(A)}|e_r(\lambda)|$. The reverse worst-case inequality follows by taking $A=z_0I_N$ and a unit vector $b$, then taking the supremum over $z_0\le0$.
\end{proof}

This order of quantifiers does not give a lower bound for an adaptive Krylov method. For a particular $(A,b)$, \eqref{eq:spectral-measure-error} samples only the occupied spectrum and its weights, so a matrix-dependent stopping test may terminate substantially earlier than the half-line design predicts; see also \cite{BeckermannGuettel2012,GuettelShao2025}.

\subsection{Shift-and-invert approximation and work bounds}
\label{subsec:shift-invert-space}

For $q>0$, the repeated shift-and-invert space is
\begin{equation}\label{eq:repeated-si-space}
 \mathcal Q_{m+1}(A,b;q)
 :=\operatorname{span}\{b,(qI-A)^{-1}b,\ldots,(qI-A)^{-m}b\}.
\end{equation}
This space is exactly $\{r(A)b:r\in\mathcal R_m(q)\}$. Indeed, a polynomial of degree at most $m$ in $(q-z)^{-1}$ is precisely a rational function $P_m(z)/(q-z)^m$ with $\deg P_m\le m$. Since the pole $q$ is fixed and repeated, all $m$ shifted solves involve the same matrix $qI-A$, so one factorization can be reused throughout. Thus degree $m$ requires $m$ applications of the shifted inverse, while $\dim\mathcal Q_{m+1}\le m+1$. With $V_{m+1}$ an orthonormal basis of this space, set
\begin{equation}\label{eq:si-projected-approximation}
 A_{m+1}=V_{m+1}^*AV_{m+1},\qquad
 u_m^{\rm SI}=V_{m+1}\exp(tA_{m+1})V_{m+1}^*b.
\end{equation}
Upon earlier invariant-subspace termination the approximation is exact and is understood to remain so at later degrees. This Rayleigh--Ritz rational Krylov approximation is used in Section~\ref{sec:numerics}; see \cite{MoretNovati2004,EshofHochbruck2006,Novati2011,Guettel2013}.

\begin{proposition}[Shift-and-invert error bounds]\label{prop:si-krylov-rate}
For every matrix $A=A^*\le0$ and vector $b\ne0$,
\begin{equation}\label{eq:si-nearoptimal}
 \|\ee^{tA}b-u_m^{\rm SI}\|_2\le2\|b\|_2\mathcal E_m^{\rm rat}(t,q).
\end{equation}
Consequently, if $q_m=\alpha m/t$ with $0<\alpha<\alpha_c$, then for any sequence $A_m=A_m^*\le0$ and $b_m\ne0$, of arbitrary dimensions,
\begin{equation}\label{eq:si-root-limsup}
 \limsup_{m\to\infty}\left(\frac{\|\ee^{tA_m}b_m-u_m^{\rm SI}\|_2}{\|b_m\|_2}\right)^{1/m}\le H_e(\alpha).
\end{equation}
At $\alpha=\alpha_*$, a constant $C>0$, independent of the matrix--vector sequence, gives
\begin{equation}\label{eq:si-optimal-strong-bound}
 \|\ee^{tA_m}b_m-u_m^{\rm SI}\|_2\le C\|b_m\|_2m^{-1/2}(\sqrt2-1)^m
\end{equation}
for all sufficiently large $m$.
\end{proposition}

\begin{proof}
Write $V=V_{m+1}$, $T=V^*AV$, and $D=(qI-A)^{-1}$. Since $D^jb$ lies in the range of $V$ for $0\le j\le m$, the vectors $c_j=V^*D^jb$ satisfy
\[
 (qI-T)c_j=V^*(qI-A)D^jb=c_{j-1},\qquad 1\le j\le m.
\]
Thus $D^jb=V(qI-T)^{-j}V^*b$, and hence $r(A)b=Vr(T)V^*b$ for every $r\in\mathcal R_m(q)$. It follows that
\[
 \ee^{tA}b-u_m^{\rm SI}
 =(\ee^{tA}-r(A))b-V(\ee^{tT}-r(T))V^*b.
\]
Both $A$ and $T$ are self-adjoint and negative semidefinite, so the spectral theorem bounds each term by the scalar uniform error times $\|b\|_2$. Taking the infimum over $r\in\mathcal R_m(q)$ proves \eqref{eq:si-nearoptimal}; Corollary~\ref{cor:rational-error-laws} gives the remaining claims.
\end{proof}

\begin{corollary}[Solve budget and stopping mismatch]\label{cor:matrix-work-budget}
There are constants $C_0\in\R$ and $\varepsilon_0>0$, independent of $A=A^*\le0$ and $b\ne0$, such that $\|\ee^{tA}b-u_m^{\rm SI}\|_2\le\varepsilon\|b\|_2$ for $0<\varepsilon<\varepsilon_0$ whenever
\begin{equation}\label{eq:matrix-work-budget}
 m\ge\left\lceil\frac{L_\varepsilon-\tfrac12\log L_\varepsilon}{\operatorname{arsinh}(1)}+C_0\right\rceil,
 \qquad q=\frac{m}{\sqrt2\,t},\qquad L_\varepsilon=\log(1/\varepsilon).
\end{equation}
Alternatively, suppose $q_N=\widehat\alpha N/t$ is fixed before iteration and the realized degree satisfies $m_N/N\to\vartheta>0$, with $0<\widehat\alpha/\vartheta<\alpha_c$. For any matrix--vector sequence $(A_N,b_N)$ as above,
\begin{equation}\label{eq:matrix-mismatch-limsup}
 \limsup_{N\to\infty}\left(\frac{\|\ee^{tA_N}b_N-u_{m_N,N}^{\rm SI}\|_2}{\|b_N\|_2}\right)^{1/m_N}
 \le H_e(\widehat\alpha/\vartheta),
\end{equation}
where $u_{m_N,N}^{\rm SI}$ is \eqref{eq:si-projected-approximation} for the $N$th problem. In particular, $\widehat\alpha=\alpha_*$ and $m_N=N+o(N)$ preserve the bound $\sqrt2-1$; the quadratic penalty for a small fixed relative mismatch is given by Corollary~\ref{cor:stopping-robustness}.
\end{corollary}

\begin{proof}
Invert \eqref{eq:si-optimal-strong-bound} as in Theorem~\ref{thm:precision-work-fixed-alpha}; the factor $2$ in \eqref{eq:si-nearoptimal}, together with other multiplicative constants, affects only $C_0$. For \eqref{eq:matrix-mismatch-limsup}, combine \eqref{eq:si-nearoptimal} with Theorem~\ref{thm:mismatch}; taking roots removes the factor \(2\).
\end{proof}

The budget is independent of mesh size even when $\|A_h\|_2\to\infty$, but factorization and solve costs need not be; finite-tolerance certification is discussed in Remark~\ref{rem:tolerance-certification}. The self-adjoint assumption is essential: for nonnormal matrices, spectral location alone does not control the matrix-function error.

\section{Numerical verification}
\label{sec:numerics}

The experiments are designed to test the refined asymptotic statements, not to benchmark a new matrix algorithm. Reference coefficients are evaluated from the horizontally deformed integral of Section~\ref{sec:uniform-saddle}, with the dominant saddle action factored out; the saddle formula itself is used only for the predicted curves. All computations use double precision. Further implementation details, including quadrature orders, tail cutoffs, angle grids, reorthogonalization, and the discrete-sine reference, are available with the numerical code upon request. The reported values are finite-precision consistency checks, not interval-certified enclosures.

\subsection{Coefficient and error asymptotics}
\label{subsec:num-coeff}
\label{subsec:num-strong-errors}

\begin{figure}[!htbp]
 \centering
 \includegraphics[width=0.68\textwidth]{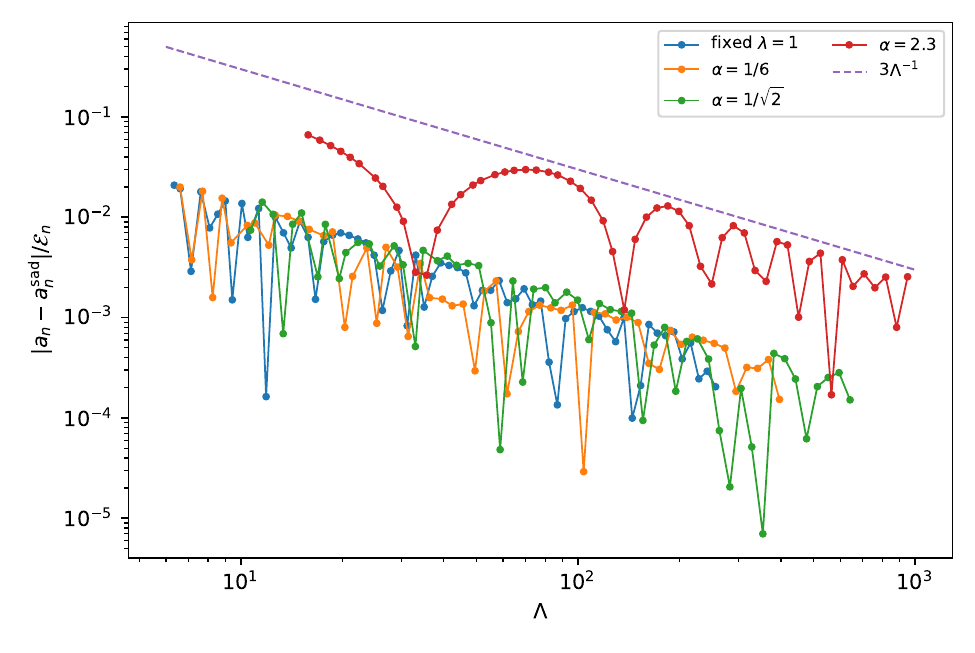}
 \caption{Normalized error of the leading two-saddle coefficient formula in the fixed-pole, optimal linear, and two additional fixed-ratio regimes. Downward spikes reflect cancellation in the oscillatory first correction. The dashed reference line is $3\Lambda^{-1}$.}
 \label{fig:num-coeff-residual}
\end{figure}

Let $a_n^{\rm sad}$ denote the leading two-saddle term in \eqref{eq:uniform-leading}, and let
\[
 \mathcal E_n=n^{-1/2}B(\alpha_n)H_e(\alpha_n)^n
\]
be the nonoscillatory envelope in \eqref{eq:coefficient-envelope}. Theorem~\ref{thm:uniform-coeff} predicts
\begin{equation}\label{eq:num-normalized-coeff-residual}
 \frac{|a_n-a_n^{\rm sad}|}{\mathcal E_n}=O(\Lambda_n^{-1}),
 \qquad
 \Lambda_n=(\lambda_n n^2)^{1/3}.
\end{equation}

Figure~\ref{fig:num-coeff-residual} tests this estimate for fixed $\lambda=1$ and for the linear ratios $\alpha=1/6$, $1/\sqrt2$, and $2.3$, the last lying near the coalescence value $3\sqrt3/2\approx2.598$. Because the first neglected saddle correction is oscillatory, the residual need not approach a smooth multiple of $\Lambda^{-1}$. The compensated residual $\Lambda_n|a_n-a_n^{\rm sad}|/\mathcal E_n$ nevertheless remains bounded in all four sequences, with maximum $2.43$ over the sampled ranges at $\alpha=2.3$, and no constant is fitted.

\paragraph{Prefactor-resolved errors.}
For the linear family $\lambda_m=\alpha m$, Proposition~\ref{prop:l2-oscillatory-profile} gives more than an $m$th-root limit: after division by $m^{-1/2}H_e(\alpha)^m$, the weighted-$L^2$ projection error follows an explicit bounded oscillatory profile. Figure~\ref{fig:num-l2-profile} compares the independently computed tail norm with the closed formula \eqref{eq:profile-function-closed}; no constant or phase is fitted. Agreement improves with $m$ for all three ratios, including $\alpha=2.3$, where convergence is slower near coalescence. The normalized error therefore remains bounded without converging to one constant.

\begin{figure}[!htbp]
 \centering
 \includegraphics[width=0.68\textwidth]{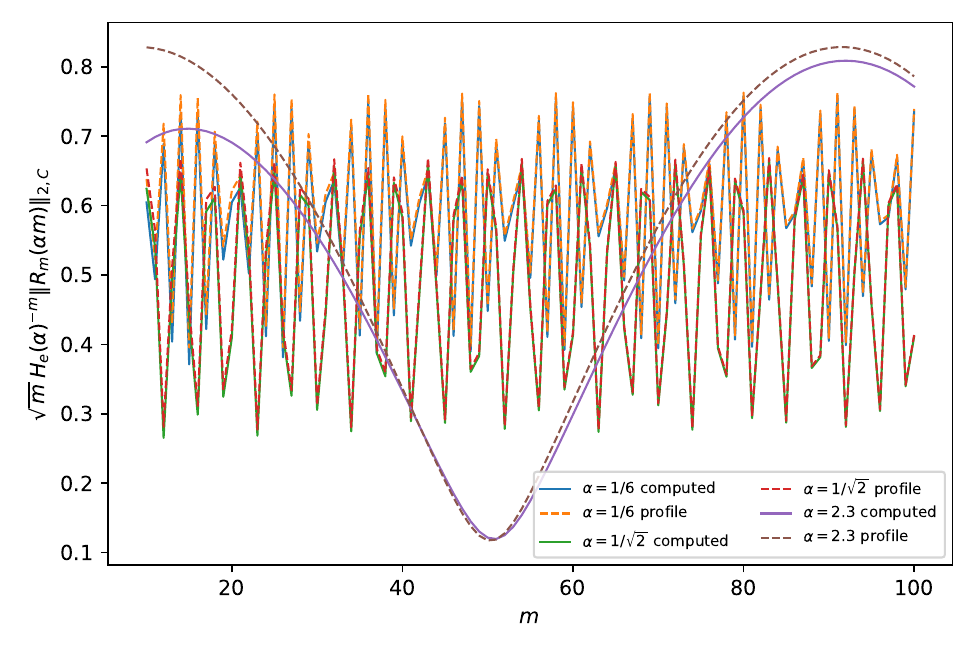}
 \caption{The normalized weighted-$L^2$ projection error (solid curves) and the explicit oscillatory profile in Proposition~\ref{prop:l2-oscillatory-profile} (dashed curves) for three exact linear pole ratios.}
 \label{fig:num-l2-profile}
\end{figure}

The projection experiment does not compute the best uniform error. Instead, Theorem~\ref{thm:approx-error-envelope} gives the rigorous bracket
\[
 \pi^{-1/2}\|R_m(\alpha m)\|_{2,C}
 \le E_m(\alpha m)
 \le \|R_m(\alpha m)\|_\infty.
\]
Figure~\ref{fig:num-uniform-bounds} evaluates the two endpoint norms for $\alpha=1/6$, $1/\sqrt2$, and $2.3$. Their bounded oscillation illustrates $E_m(\alpha m)\asymp m^{-1/2}H_e(\alpha)^m$; the figure does not identify the projection error with the minimax error.

\begin{figure}[!htbp]
 \centering
 \includegraphics[width=0.68\textwidth]{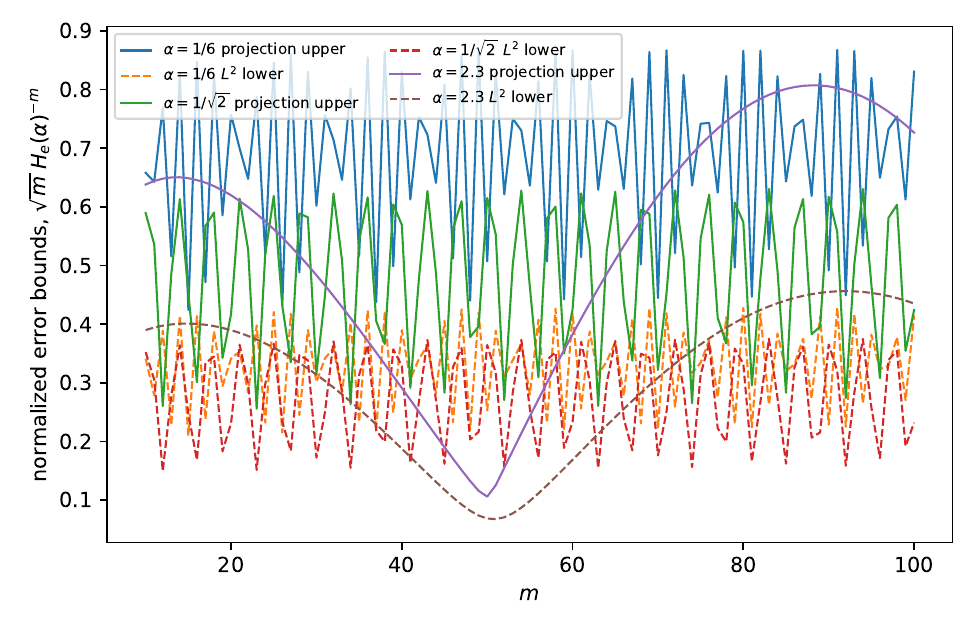}
 \caption{Numerical evaluations of the normalized endpoints in the rigorous best-error bracket at three exact linear pole ratios, including $\alpha=1/\sqrt2$. The dashed curves approximate $\pi^{-1/2}\|R_m\|_{2,C}$ and the solid curves approximate the Chebyshev projection errors; the corresponding exact endpoint norms bracket the unknown best error.}
 \label{fig:num-uniform-bounds}
\end{figure}

\subsection{Transition and work laws}
\label{subsec:num-transition}
\label{subsec:num-design}

Table~\ref{tab:num-transition} tests the logarithmic and root limits in Corollary~\ref{cor:error-transition}. For fixed $\lambda=1$ the predicted limit is $3/2$. For $\lambda_m=\frac18m^{1/2}$ it is $\frac32(1/8)^{1/3}=3/4$. In the optimal linear case, compensating by $\sqrt m$ isolates the root factor $\sqrt2-1$.

\begin{table}[!htbp]
 \centering
 \caption{Scaled Chebyshev projection errors in the fixed, sublinearly moving, and linearly moving regimes.}
 \label{tab:num-transition}
 \footnotesize
 \renewcommand{\arraystretch}{1.15}
 \begin{tabular}{@{}llllc@{}}
 \toprule
 Regime & scaled quantity & degrees $m$ & observed values & limit \\
 \midrule
 $\lambda=1$
 & $-m^{-2/3}\log\|R_m\|_\infty$
 & $80,320,1280$ & $1.5594,1.5371,1.5166$ & $1.5000$ \\
 $\lambda_m=\frac18m^{1/2}$
 & $-m^{-5/6}\log\|R_m\|_\infty$
 & $80,320,1280$ & $0.7922,0.7578,0.7501$ & $0.7500$ \\
 $\lambda_m=m/\sqrt2$
 & $(\sqrt m\,\|R_m\|_\infty)^{1/m}$
 & $80,320,640$ & $0.4113,0.4135,0.4139$ & $0.4142$ \\
 \bottomrule
 \end{tabular}
 \normalsize
\end{table}

The sublinear and linear values are already close to their limits at the largest displayed degrees. The fixed-$\lambda$ convergence is visibly slower because an algebraic prefactor contributes an $O(\log m/m^{2/3})$ correction after the scaled logarithm is taken; the remaining offset at $m=1280$ matches this correction.

\paragraph{Tolerance-driven work and stopping uncertainty.}
We next test the inverse design law at $t=1$ and $\alpha_*=1/\sqrt2$. Table~\ref{tab:num-precision-work} uses the computable Chebyshev projection error rather than the unknown best error. By Theorem~\ref{thm:approx-error-envelope}, both have the same $m^{-1/2}\ee^{-\kappa_*m}$ two-sided scale, so their first-passage degrees obey the same two-term law up to $O(1)$. The table compares the observed projection first passage with $m^{(1)}_\varepsilon=\log(1/\varepsilon)/\operatorname{arsinh}(1)$ and $m^{(2)}_\varepsilon=m^{(1)}_\varepsilon-\log\log(1/\varepsilon)/(2\operatorname{arsinh}(1))$. No fitted constant is used; the logarithmic correction removes most of the finite-degree bias.

\begin{table}[!htbp]
 \centering
 \caption{Precision-to-work prediction at the optimal pole ratio. The observed degree is the first $m$ for which the Chebyshev projection error is at most $\varepsilon$.}
 \label{tab:num-precision-work}
 \footnotesize
 \renewcommand{\arraystretch}{1.12}
 \begin{tabular}{ccccc}
 \toprule
 $\varepsilon$ & observed $m$ & $m^{(1)}_\varepsilon$ & $m^{(2)}_\varepsilon$ & observed error \\
 \midrule
 $10^{-4}$  & $9$  & $10.45$ & $9.19$  & $4.47\times10^{-5}$ \\
 $10^{-6}$  & $14$ & $15.67$ & $14.19$ & $7.17\times10^{-7}$ \\
 $10^{-8}$  & $19$ & $20.90$ & $19.25$ & $4.19\times10^{-9}$ \\
 $10^{-10}$ & $23$ & $26.12$ & $24.35$ & $8.38\times10^{-11}$ \\
 $10^{-12}$ & $29$ & $31.35$ & $29.47$ & $8.58\times10^{-13}$ \\
 \bottomrule
 \end{tabular}
 \normalsize
\end{table}

The local robust-ratio expansion is also accurate at moderate uncertainty. For $\vartheta\in[0.9,1.1]$, the exact equalizing ratio is $0.702053$, compared with $0.702063$ from \eqref{eq:robust-alpha-expansion}; for $[0.8,1.2]$, the values are $0.686780$ and $0.686931$. The improvement in the worst root factor is modest, consistent with quadratic flatness.

\subsection{A matrix-action illustration}
\label{subsec:num-matrix}

\begin{figure}[!htbp]
 \centering
 \includegraphics[width=0.68\textwidth]{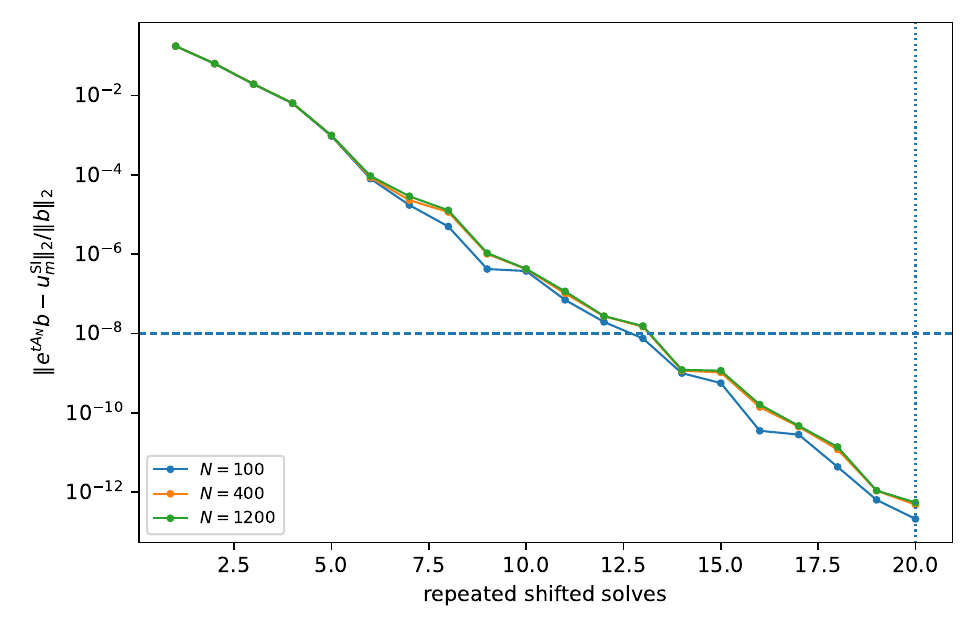}
 \caption{Input-normalized errors $\|\ee^{tA_N}b-u_m^{\rm SI}\|_2/\|b\|_2$ of the projected shift-and-invert Krylov approximation \eqref{eq:si-projected-approximation} for three grid sizes. The same repeated pole is used throughout; the dashed and dotted lines mark the tolerance and nominal budget.}
 \label{fig:num-matrix-action}
\end{figure}

Finally, consider the negative Dirichlet finite-difference Laplacian $A_N=(N+1)^2\operatorname{tridiag}(1,-2,1)$ on $(0,1)$. We take $t=10^{-2}$, $\varepsilon=10^{-8}$, and the normalized deterministic vector with entries proportional to $\exp[-100(x_j-0.35)^2]$, $x_j=j/(N+1)$. The two-term law gives $m^{(2)}_\varepsilon=19.25$, so we choose the nominal budget $20$ and $q=20/(\sqrt2\,t)=1414.2136$.
At each step we use the projected approximation \eqref{eq:si-projected-approximation}; one factorization of $qI-A_N$ is reused, and the reference action is evaluated by the exact discrete-sine diagonalization. Figure~\ref{fig:num-matrix-action} shows $\|\ee^{tA_N}b-u_m^{\rm SI}\|_2/\|b\|_2$ for $N=100,400,1200$. Since $b$ is normalized, these are also the absolute errors reported by the code; they are not normalized by $\|\ee^{tA_N}b\|_2$. The tolerance is reached after $13$, $14$, and $14$ repeated solves, respectively, while the errors at $m=20$ are $2.12\times10^{-13}$, $4.84\times10^{-13}$, and $5.48\times10^{-13}$. Thus the same half-line pole and nominal budget remain effective under mesh refinement, whereas the particular spectral measures permit earlier stopping than the matrix-independent worst-case design.

\section{Concluding remarks}
\label{sec:conclusion}

The paper connects three scales usually treated separately. A fixed repeated pole gives the flat-endpoint stretched exponential, a pole proportional to the degree gives Andersson's geometric regime, and one uniform two-saddle formula describes the transition between them before coalescence. The technical point that makes this possible is uniform control of the amplitude, phase, and relative saddle remainder as the pole-to-index ratio tends to zero.

Coefficient asymptotics alone are not enough: two saddle contributions can cancel at isolated indices. The growing phase block shows that such cancellation cannot suppress the whole relevant tail, producing the prefactor-resolved two-sided law $E_m(\alpha m)\asymp m^{-1/2}H_e(\alpha)^m$ and an explicit bounded oscillatory profile for the weighted $L^2$ projection error.

These estimates make the pole law operational, as they give a two-term precision-to-work formula, quantify degree--pole mismatch, and lead to a simple robust design under stopping uncertainty. For self-adjoint negative semidefinite matrices, the scalar restricted-denominator error is a worst-case matrix-action benchmark for fixed rational functions, while the corresponding shift-and-invert Krylov approximation inherits dimension-independent bounds.

Two natural asymptotic problems remain outside the present analysis.  Passing through saddle coalescence requires an Airy-type local model, and for a fixed physical pole the present analysis does not yet provide a sharp algebraic prefactor in the uniform norm; there the weighted lower bound and absolute-tail upper bound do not yet meet.  Nonnormal matrices require a different approximation geometry and are likewise beyond the scope of the half-line result proved here.

\appendix

\section{Technical details for the uniform saddle estimates}
\label{app:saddle-details}

\subsection{Horizontal saddle contour}
\begin{proof}
Using $w^2+\delta^2=(w+i\delta)(w-i\delta)$, the exact scaled integrand is
\[
 \delta\,\ee^{-\Lambda w^2}
 \frac{(w+i\delta)^{n-1}}{(w-i\delta)^{n+1}}.
\]
Its only pole is at $w=i\delta$, so the strip $-v\le\Ima w\le0$ is pole-free. On vertical sides of a rectangular deformation the Gaussian contributes $\exp[-\Lambda((\Rea w)^2-(\Ima w)^2)]$, while the remaining factor is algebraic; hence the vertical contributions vanish and the contour may be shifted to $\R-iv$.

Taking moduli gives \eqref{eq:Gdelta}, and differentiation gives \eqref{eq:Gprime}. At $x=u$, the saddle relations \eqref{eq:uv-relations} imply
\[
 [u^2+(v-\delta)^2][u^2+(v+\delta)^2]=4v^2(4v^2-\delta^2)=2v,
\]
so the bracket in \eqref{eq:Gprime} vanishes there. Its denominator is strictly increasing in $x^2$, which yields the asserted sign pattern and makes $\pm u$ the only global maxima; the exceptional point $v=\delta$, $x=0$ is the regular point discussed in Remark~\ref{rem:removable-zero}. The limit \eqref{eq:G0} follows from \eqref{eq:delta-zero-phase-amplitude} and $v(\delta)\to1/2$.

Finally, \eqref{eq:real-second-derivative} is uniformly negative at the saddles on $0\le\delta\le\delta_0$, and $u(\delta)$ stays uniformly positive. A fixed quadratic neighborhood therefore has a parameter-independent action drop. Compactness extends this drop over every bounded part of the contour outside the two saddle neighborhoods. On the tails, expansion of \eqref{eq:Gdelta} gives $G_\delta(x)=-x^2+O(1)$ uniformly in $0\le\delta\le\delta_0$ (with \eqref{eq:G0} at $\delta=0$). This proves the uniform gap \eqref{eq:uniform-gap}.
\end{proof}

\subsection{Uniform local saddle estimate}
\begin{proof}
Since $\delta_0<\delta_c$, the compact saddle curve $w_+(\delta)$ stays away from the origin and $\pm i\delta$. Thus $\Phi_\delta$ and $A_\delta$ are uniformly analytic in one fixed neighborhood of the curve. By \eqref{eq:real-second-derivative}, $\Rea\Phi_\delta''(w_+)\le-c_0<0$ uniformly; after fixing a sufficiently small $\eta$, Taylor's theorem gives
\[
 \Rea\{\Phi_\delta(w_++s)-\Phi_\delta(w_+)\}\le-c_1s^2,
 \qquad |s|\le2\eta,
\]
with $c_1$ independent of $\delta$.

Set $s=\Lambda^{-1/2}\xi$. Uniform Taylor expansions of phase and amplitude give, on $|\xi|\le\Lambda^{1/12}$,
\[
 A_\delta(w_++s)\ee^{\Lambda[\Phi_\delta(w_++s)-\Phi_\delta(w_+)]}
 =\ee^{\Phi_\delta''(w_+)\xi^2/2}
 \left[A_\delta(w_+)+\Lambda^{-1/2}P_\delta(\xi)
 +O\!\left(\Lambda^{-1}(1+|\xi|^6)\right)\right],
\]
where $P_\delta$ is odd and the remainder is uniform. Since the cutoff is even, the $\Lambda^{-1/2}$ term integrates to zero. The quadratic bound makes the displayed remainder uniformly integrable and renders the rest of the cutoff region exponentially small. Extending the leading Gaussian to $\R$ therefore yields \eqref{eq:uniform-local-asymptotic} with an absolute relative correction $O(\Lambda^{-1})$. Finally, the saddle equation gives $A_\delta(w_+)=iw_+$, which is bounded away from zero on the compact saddle curve; the absolute estimate is therefore the uniform relative bound \eqref{eq:uniform-local-remainder}.
\end{proof}

\section{Local expansion for the robust stopping design}
\label{app:robust-expansion}

\begin{proof}[Corollary~\ref{cor:narrow-robust}]
Let $h=\log H_e$ and $y(\epsilon)=\widehat\alpha_{\rm rob}/\vartheta_0$. Endpoint equalization becomes
\[
 h\!\left(\frac{y}{1-\epsilon}\right)
 =h\!\left(\frac{y}{1+\epsilon}\right).
\]
After division by $2\epsilon$, the difference extends analytically and evenly to $\epsilon=0$, where it equals $yh'(y)$. Since $h'(\alpha_*)=0$ and $\alpha_*h''(\alpha_*)\ne0$, the implicit-function theorem gives a unique even analytic branch $y(\epsilon)$ through $\alpha_*$. Hence
$y(\epsilon)=\alpha_*[1+c_2\epsilon^2+O(\epsilon^4)]$.

From \eqref{eq:zeta-alpha-derivative}, one more differentiation at $\alpha_*=1/\sqrt2$ and $\zeta_{\alpha_*}=e^{i\pi/4}$ gives
\[
 \zeta_{\alpha_*}'=-\frac{1+3i}{5},\qquad
 \zeta_{\alpha_*}''=\frac{\sqrt2(38+84i)}{125},\qquad
 h'''(\alpha_*)=-\frac{172}{125}.
\]
Expanding the equalization equation through order $\epsilon^3$ yields
\[
 2\alpha_*^2h''(\alpha_*)(1+c_2)
 +\frac13\alpha_*^3h'''(\alpha_*)=0,
\]
so $c_2=-107/150$ and \eqref{eq:robust-alpha-expansion} follows. Substitution into either equalized endpoint, followed by Taylor expansion at $\alpha_*$, gives \eqref{eq:robust-factor-expansion}; even analyticity supplies the $O(\epsilon^4)$ remainder.
\end{proof}

\section*{Declarations}

\paragraph{Funding}
The first author was supported in part by the National Science Foundation under grant DMS-2608905. 

\paragraph{Author contributions}
Both authors contributed to the conception of the study, mathematical analysis, numerical verification, and preparation of the manuscript. Both authors reviewed and approved the final manuscript.

\paragraph{Data and code availability}
The numerical data and code used to produce the results in Section~\ref{sec:numerics} are available from the authors upon reasonable request.

\paragraph{Competing interests}
The authors declare that they have no competing interests.

\paragraph{Use of artificial intelligence tools}
OpenAI's ChatGPT, including GPT-5.6 and GPT-6 models, were used to suggest candidate derivations, assist with algebraic and consistency checks, compare exposition with relevant literature, and review numerical code. They were also used for language and presentation improvements. All mathematical arguments and computational results reported in the manuscript were independently checked and verified by the authors. All citations were verified against the original sources. The authors take full responsibility for the content of the manuscript.

\end{document}